\documentclass[12pt,reqno]{amsart}
\usepackage[centering, includeheadfoot, hmargin=1.0in, tmargin=1.0in, bmargin=1in, headheight=7pt]{geometry}
\usepackage[alphabetic]{amsrefs}
\usepackage{url}

\usepackage{booktabs}
\usepackage[utf8]{inputenc}
\usepackage{graphicx}
\usepackage{color}
\usepackage{amsmath}
\usepackage[noabbrev,capitalize]{cleveref}
\usepackage{mathrsfs}
\usepackage{stmaryrd}
\usepackage{mathtools}
\usepackage{amssymb}
\usepackage[pagewise]{lineno}
\usepackage{appendix}
\usepackage{tikz-cd}
\newtheorem{lemma}{Lemma}[section]
\newtheorem{theorem}[lemma]{Theorem}
\newtheorem{corollary}[lemma]{Corollary}
\newtheorem{proposition}[lemma]{Proposition}
\theoremstyle{definition}
\newtheorem{definition}[lemma]{Definition}

\newtheorem{question}[lemma]{Question}

\newtheorem{conjecture}[lemma]{Conjecture}

\theoremstyle{remark}
\newtheorem{remark}[lemma]{Remark}

\numberwithin{table}{section}

\newcommand{\PP}{\mathbb P}

\newcommand{\ZZ}{\mathbb Z}

\newcommand{\QQ}{\mathbb Q}

\newcommand{\CC}{\mathbb C}
\newcommand{\RR}{\mathbb R}

\newcommand{\frq}{\mathfrak q}

\newcommand{\frp}{\mathfrak p}

\newcommand{\cS}{\mathcal S}

\newcommand{\cH}{\mathcal H}
\newcommand{\cI}{\mathcal I}
\newcommand{\cO}{\mathcal O}
\newcommand{\cL}{\mathcal L}

\newcommand{\ba}{\mathbf a}
\newcommand{\bb}{\mathbf b}
\newcommand{\bc}{\mathbf c}
\newcommand{\bd}{\mathbf d}

\newcommand{\bu}{\mathbf u}
\newcommand{\bv}{\mathbf v}
\newcommand{\bw}{\mathbf w}
\newcommand{\bx}{\mathbf x}
\newcommand{\by}{\mathbf y}

\newcommand{\eps}{\varepsilon}

\DeclareMathOperator{\codim}{codim}

\DeclareMathOperator{\GL}{GL}

\DeclareMathOperator{\Proj}{Proj}

\DeclareMathOperator{\reg}{reg}

\newcommand{\lra}{\DOTSB\longrightarrow}
\newcommand{\embed}{\hookrightarrow}
\newcommand{\abs}[1]{\lvert#1\rvert}
\newcommand{\llb}{\llbracket}
\newcommand{\rrb}{\rrbracket}

 \numberwithin{equation}{section}

\begin{document}
\title[Projective toric varieties of codimension 2 with maximal regularity]{Projective toric varieties of codimension 2 with maximal Castelnuovo--Mumford regularity}

\author{Preston Cranford}
\address{Massachusetts Institute of Technology, Cambridge, MA 02139, USA}
\email{prestonc@mit.edu}

\author{Alan Peng}
\address{Massachusetts Institute of Technology, Cambridge, MA 02139, USA}
\email{apeng1@mit.edu}

\author{Vijay Srinivasan}
\address{Massachusetts Institute of Technology, Cambridge, MA 02139, USA}
\email{vijayrs@mit.edu}

\date{June 2021}
\subjclass[2020]{}

\begin{abstract}
The Eisenbud--Goto conjecture states that $\operatorname{reg} X\le\operatorname{deg} X -\operatorname{codim} X+1$ for a nondegenerate irreducible projective variety $X$ over an algebraically closed field. While this conjecture is known to be false in general, it has been proven in several special cases, including when $X$ is a projective toric variety of codimension $2$. We classify the projective toric varieties of codimension $2$ having maximal regularity, that is, for which equality holds in the Eisenbud--Goto bound. We also give combinatorial characterizations of the arithmetically Cohen--Macaulay toric varieties of maximal regularity in characteristic $0$.
\end{abstract}

\maketitle

\section{Introduction}
\label{sec:introduction}
Let $\Bbbk$ be a field. Recall that for a closed subscheme $X\embed \PP^n_\Bbbk$ with ideal sheaf $\cI_X$, the \textit{Castelnuovo--Mumford regularity} of $X$ is defined as $$\reg X \coloneqq \max\{r > 0\mid H^i(\PP^n_\Bbbk, \cI_X(r-i)) = 0 \text{ for all }i>0\}.$$ Equivalently, $\reg X$ may be defined in terms of the minimal free resolution of $\cI_X$, viewed as an $\cO_{\PP^n}$-module (see $\S2$ for this definition). In \cite{EG}, Eisenbud and Goto made the following conjecture.

\begin{conjecture}\label{EGconj}
If $\Bbbk$ is an algebraically closed field and $X\embed \PP^n_\Bbbk$ is a nondegenerate irreducible projective variety, then
\[
\reg X\leq\deg X-\codim X+1.
\]
\end{conjecture}
In 2018, Peeva and McCullough proved that this conjecture is false in general by exhibiting an explicit counterexample \cite{MP}.

While not true in full generality, \cref{EGconj} has been proven true in several special cases. For example, in \cite{EG}, Eisenbud and Goto proved this bound holds if $X$ is arithmetically Cohen--Macaulay. In \cite{GLP}, \cref{EGconj} is proven in the case where $X$ is a projective curve, using algebro-geometric methods. However, it is still open whether \cref{EGconj} is true for general \textit{toric varieties}, which are the subject of focus of the present paper. In this direction, \cite{HH} resolved the case where the coordinate ring of $X$ is a simplicial semigroup ring with isolated singularity, and \cite{nitsche12} resolved the case where the coordinate ring is a seminormal simplicial affine semigroup ring. A combinatorial proof of \cref{EGconj} in the case of monomial curves was given in \cite{nitsche14}. In \cite{PS}, a combinatorial proof was given for toric varieties of codimension $2$.

In the cases of (nondegenerate, irreducible) curves, arithmetically Cohen--Macaulay varieties, and toric varieties of codimension $2$ (each a large class of varieties for which \cref{EGconj} is known to hold), we will use the phrase \emph{having maximal regularity} to refer to varieties $X$ which achieve equality in \cref{EGconj}, i.e. such that $\reg X = \deg X - \codim X + 1$. The main question motivating the results of this paper is the following.

\begin{question}
Which toric varieties have maximal regularity (in the cases where the Eisenbud--Goto bound is known to hold)?
\end{question}

We work over an arbitrary field $\Bbbk$. Let $S$ denote the coordinate ring of $\PP^{n-1}_\Bbbk$. To a lattice $\cL \subseteq \ZZ^n$ of rank $r$ perpendicular to $(1,\dots,1)\in\ZZ^n$, there is an associated homogeneous \textit{lattice ideal} $I_\cL\subseteq S$ and a closed projective scheme $X_\cL\coloneqq \Proj S/I_\cL\embed \PP^{n-1}_\Bbbk$. Also associated to $\cL$ is its \textit{Gale diagram}, which is a certain collection of $n$ vectors in $\ZZ^r$ that is defined uniquely up to the action of $\GL_r(\ZZ)$. When $\cL$ is \textit{saturated} (that is, $\QQ\cL\cap\ZZ^n=\cL$), then $I_\cL$ is prime and is a \textit{toric ideal}, and $X_\cL$ is a (projective) \textit{toric variety}. The $1$-dimensional toric varieties are precisely the monomial curves.

Our main result addresses the general case when $\codim X_\cL=2$, and is stated below. Recall that a nonzero lattice point $\bu\in\ZZ^2$ is \textit{visible} if its coordinates are relatively prime.

\begin{theorem}\label{thm:main}
Suppose $X_\cL$ is a nondegenerate toric variety of codimension $2$. Let $n'\ge 3$ be the number of nonzero Gale vectors of $\cL$. Then $X_\cL$ has maximal regularity (i.e., achieves equality in the Eisenbud--Goto conjecture) if and only if $\cL$ satisfies one of the following:
\begin{itemize}
    \item After removing Gale vectors equal to $0$, the lattice $\cL$ is one of the $14$ saturated lattices given in \cref{ci-lattice-ideals}, up to permutations of the coordinates. In this case, $4\le n'\le 8$, and $X_\cL$ is a complete intersection.
    
    \item $n'=4$, and $\cL$ has a Gale diagram whose nonzero vectors are equal to
    \[\{(1,0),(-1,1),(-1,-d+1),(1,d-2)\}\]
    for some $d\ge 3$. In this case, $X_\cL$ is not a complete intersection, and $X_\cL$ is arithmetically Cohen--Macaulay if and only if $d=3$.
    
    \item $n'=5$, and the Gale diagram of $\cL$, ignoring zeros, is of the form
    \[\{\bu,\bv,\bw,-\bu,-\bv-\bw\},\]
    where $\bu,\bv,\bw$ are visible lattice points that are pairwise linearly independent and are such that $\bu\neq\pm(\bv+\bw)$ and $\abs{\det(\bv,\bw)}=1$. In this case, $X_\cL$ is not a complete intersection.
    
    \item $n'=6$, and the Gale diagram of $\cL$, ignoring zeros, is of the form
    \[\{\bu,\bv,\bw,-\bu,-\bv,-\bw\},\]
    where $\bu,\bv,\bw$ are visible lattice points that are pairwise linearly independent and are such that $\abs{\det(\bv,\bw)}=1$. In this case, $X_\cL$ is not a complete intersection.
\end{itemize}
\end{theorem}

Note that in the third case, the conditions given on the Gale diagram are indeed $\GL_2(\ZZ)$-invariant. Moreover, the set of saturated lattices given in \cref{ci-lattice-ideals} is independent of the field $\Bbbk$. 

\begin{remark}
\label{rem:explicitdegrees}
A projective toric subvariety of $\PP_{\Bbbk}^{n-1}$ of codimension $r$ may also be specified by a linear map $A : \ZZ^n\to \ZZ^{n-r}$; we set $X_A \coloneqq X_\cL$ where $\cL = \ker A$. For simplicity assume $n=n'$, i.e. that all Gale vectors are nonzero. Then the toric varieties of \cref{thm:main} which are not complete intersections can be described as the varieties $X_A$ where $A$ takes one of the following forms (for $n=4,5,6$, respectively):
\[
\begin{bmatrix} 
1 & 1 & 1 & 1\\ 
0 & 1 & d-1 & d 
\end{bmatrix},
\qquad 
\begin{bmatrix} 
1 & 0 & 0 & 1 & 0\\ 
0 & 1 & 1 & 0 & 1\\
1 & b & c & 0 & 0
\end{bmatrix},
\qquad
\begin{bmatrix} 
1 & 0 & 0 & 1 & 0 & 0\\ 
0 & 1 & 0 & 0 & 1 & 0\\
0 & 0 & 1 & 0 & 0 & 1\\
1 & b & c & 0 & 0 & 0
\end{bmatrix}
\]
where $d\ge 3$ and $b,c$ are coprime nonzero integers (and in the case $n=5$, $(b,c)\ne\pm(1,1)$). This description is equivalent to the one in \cref{thm:main}. It is easy to compute the degrees of the corresponding varieties to be $$d,\qquad  1+\max\{|b|,|c|,|b-c|\}, \qquad 1+|b|+|c|,$$ respectively. (Note that $b$ and $c$ are the unique integers for which $\bu = b\bv + c\bw$.)
\end{remark}

Our result shows that all examples of maximal regularity for toric varieties of codimension 2, in a suitable sense, come from examples in $\PP^n$ for $n\le 5$. In \cite{KW}, Kwak shows that the only smooth nondegenerate complex threefolds in $\PP^5$ with $\reg X \le \deg X - 1$ are complete intersections of two quadrics or the Segre threefold (obtained by the Segre embedding $\PP^2\times \PP^1\embed \PP^5$). It may be verified explicitly that the complete intersections appearing in \cref{thm:main} are not smooth, so it follows from Kwak's results that the only \textit{smooth} example of a toric threefold in $\PP_\CC^5$ having maximal regularity is given by the Segre threefold. A corollary of \cref{thm:main} is the following stronger fact.
\begin{theorem}
\label{smoothness}
A smooth nondegenerate toric variety of codimension $2$ (a priori of arbitrary dimension) having maximal regularity is, up to isomorphism, one of the following:
\begin{itemize}
    \item a monomial curve (embedded in $\PP^3$),
    \item a hyperplane section of the Segre threefold (embedded in $\PP^4$), or 
    \item the Segre threefold (embedded in $\PP^5$).
\end{itemize}
\end{theorem}

This hyperplane section of the Segre threefold is obtained by substituting $(b,c)=(1,-1)$ in the $3\times 5$ matrix of \cref{rem:explicitdegrees}, and the Segre threefold itself is obtained by substituting $(b,c)=(1,\pm 1)$ in the $4\times 6$ matrix. We omit the proof of this theorem; the result may be deduced by elementary means from the explicit form of the matrices in \cref{rem:explicitdegrees} (in all cases there, except those listed in \cref{smoothness}, the varieties $X_A$ are not even normal) and the explicit lattices of \cref{ci-lattice-ideals}. \\ 


Along the way to proving \cref{thm:main}, we also characterize toric varieties $X_\cL$ of maximal regularity in the cases when $X_\cL$ is a curve, a complete intersection, or arithmetically Cohen--Macaulay in characteristic $0$.

While our main results apply only to toric varieties, many of our intermediate results apply more generally to nonreduced schemes $X_\cL$ (or occasionally to even more general classes of schemes), and so we will state intermediate results as generally as the methods allow.

\subsection*{Outline}
Preliminaries and notation are discussed in \cref{sec:preliminaries}. In \cref{sec:curves}, we characterize monomial curves of maximal regularity. Specializing this result to the case of codimension $2$ allows us to solve the $n'=4$ case of \cref{thm:main}. In \cref{sec:CM}, we focus on the case when $I_\cL$ is Cohen--Macaulay; that is, when $X_\cL$ is arithmetically Cohen--Macaulay. 

In \cref{sec:PSrecall}, we recall the methods of \cite{PS}, which allow us to extract information about general schemes $X_\cL$ of codimension $2$ by understanding the case where $X_\cL \embed \PP^3$ is a curve. The bulk of the work lies in \cref{sec:analyzingreduction}, in which we use this framework to analyze the case when $X_\cL$ has codimension $2$ and is not arithmetically Cohen--Macaulay. This allows us to complete the proof of \cref{thm:main} in \cref{sec:finalproof}. We suggest directions for future research in \cref{sec:future-work}.

\subsection*{Acknowledgments}
We would like to thank Irena Peeva for suggesting this topic. We owe a great deal to Christine Berkesch and Mahrud Sayrafi for their mentorship and support. We also thank the anonymous referee for helpful comments and for pointing out the reference \cite{KW}. This research began at the 2020 University of Minnesota, Twin Cities REU, supported by NSF RTG grant DMS-1745638.

\section{Preliminaries}\label{sec:preliminaries}
Let $\Bbbk$ be a field, and let $S_n\coloneqq \Bbbk[\bx] = \Bbbk[x_1,\dots,x_n]$ be a graded polynomial ring (with the standard grading). When $n$ is fixed or clear from context, we will simply write $S$ for $S_n$. All ideals of $S$ are assumed to be homogeneous, unless stated otherwise; in the same vein, all generators of ideals are assumed to be homogeneous. Given an ideal $I\subseteq S_n$, we write $\deg I$ for the multiplicity of $S_n/I$. 

For a graded $S$-module $M$, we denote the graded minimal free resolution of $M$
$$\cdots \lra F_2\lra F_1\lra F_0 \lra M\lra 0$$ by $F_\bullet$. Each $F_i$ is a direct sum of twists of $S$, so we can write $F_i = \bigoplus_{j\in\ZZ} S(-j)^{\beta_{i,j}}$ for some nonnegative integers $\beta_{i,j}$, called the \textit{Betti numbers} of $F_\bullet$. 

The \textit{Castelnuovo--Mumford regularity} of a finitely generated graded $S$-module $M$ is defined by $$\reg M \coloneqq \max\{j\mid\beta_{i,i+j}\ne 0 \text{ for some } i\}.$$ It does not depend on the choice of minimal free resolution. For a closed subscheme $X\embed \PP^{n-1}_\Bbbk$, we write $\reg X$ for the Castelnuovo--Mumford regularity of its vanishing ideal; this is equivalent to the definition of $\reg X$ given in \cref{sec:introduction} but is more suited to our purposes. We will often refer to this quantity simply as the ``regularity.''

Because we are primarily concerned with toric varieties, we include some background on lattice ideals. Let $\cL\subseteq\ZZ^n$ be a lattice of rank $r\ge 1$ that is orthogonal to the all-$1$'s vector $(1,1,\dots,1)\in\ZZ^n$, so $r\le n-1$. Then $\cL$ defines a homogeneous \textit{lattice ideal} $$I_\cL \coloneqq \langle \bx^\bu - \bx^\bv \mid \bu - \bv \in \cL \rangle\subseteq S.$$ The codimension of $I_\cL$ equals $r$. We say that $\cL$ is \emph{saturated} if $\QQ \cL \cap \ZZ^n = \cL$. In this case, $I_\cL$ is prime and we say that $I_\cL$ is \textit{toric}.

Let $B\in\ZZ^{n\times r}$ be a matrix whose columns form a basis for $\cL$, whose entry in the $(i,j)$ position is given by $b_{ij}$. The \textit{Gale diagram} $G\coloneqq G_B$ of $\cL$ is the collection of \textit{Gale vectors}, which are the row vectors $\bb_i\coloneqq (b_{i1},\dots,b_{ir})\in\ZZ^r$. Note that for two such matrices $B$ and $B'$, the corresponding Gale diagrams differ only by an element of $\GL_r(\ZZ)$. We may therefore speak of \textit{the} Gale diagram $G_\cL \in \ZZ^{n\times r}/\GL_r(\ZZ)$ of $\cL$, independent of the choice of $B$. One can often move back and forth between the combinatorial data contained in the Gale diagram and the algebraic properties of a lattice ideal. Note that Gale vectors should be thought of as ordered pairs $(i,\bb_i)$ (so that particular vectors $\bb_i\in\ZZ^r$ are considered ``with multiplicity''), but for convenience, we will frequently treat each Gale vector $(i,\bb_i)$ simply as a vector $\bb_i\in\ZZ^r$.

Given an ideal $I\subseteq S_n$ (not necessarily prime), one can form for each $m\ge n$ the ideal $I_m\coloneqq I S_m\subseteq S_m$ (under the natural embedding $S_n\embed S_m$). It is easy to argue that for each $m \ge n$, $$\reg I_m = \reg I, \qquad \deg I_m = \deg I, \quad\text{and}\quad \codim I_m = \codim I.$$  If $I=I_\cL\subseteq S_n$ is a lattice ideal, then the ideal $I_m$ is the lattice ideal whose Gale diagram is obtained by adding $m-n$ zero vectors to the Gale diagram of $I_n$. Thus, as far as regularity, degree, and codimension are concerned, it suffices to consider lattice ideals $I_\cL$ all of whose Gale vectors are nonzero. 

Given $\cL$, we let $\Gamma\coloneqq \ZZ^n/\cL$ be an abelian group. Note that $S$, $I_\cL$, $S/I_\cL$, the minimal free resolution of $S/I_\cL$ over $S$, and the Koszul homology of $S/I_\cL$ are graded by $\Gamma$ \cite{PS}. A \textit{fiber} is a set of all monomials of $S$ with a fixed degree $C\in\Gamma$. The requirement that $I_\cL$ be orthogonal to $(1,\dots,1)$ ensures that all fibers are finite.

In the case that $\cL$ has rank $r=2$, there are more precise results on when $I_\cL$ is toric, a complete intersection, or Cohen--Macaulay. In this case, $\codim I_\cL=2$ and a Gale diagram $G$ consists of vectors in $\ZZ^2$. Following \cite{PS}, we say a Gale diagram is \textit{imbalanced} if every Gale vector lies on the $y$-axis or has nonpositive $y$-coordinate. We have the following:

\begin{lemma}[\cite{PS}*{Lemma 3.1}]\label{lem:ci-imbalanced}
The lattice ideal $I_\cL$ is a complete intersection if and only if it has an imbalanced Gale diagram.
\end{lemma}

In particular, this result implies that if $I_\cL$ is degenerate, i.e., if $I_\cL\not\subseteq\langle x_1,\dots,x_n\rangle^2$, then $I_\cL$ is a complete intersection. Moreover, we have the following result.

\begin{proposition}[\cite{PS}*{Proposition 4.1}]\label{prop:non-cm-open-quadrants}
The ideal $I_\cL$ is not Cohen--Macaulay if and only if there exists a Gale diagram $G$ for $\cL$ which intersects each of the four open quadrants.
\end{proposition}

\section{Monomial curves}
\label{sec:curves}

Let $A$ be a $2\times n$ matrix of the form 
\begin{equation}
\label{matrixform}
    \begin{bmatrix} 1 & 1 & \cdots & 1\\ a_1 & a_2 & \cdots & a_n \end{bmatrix},
\end{equation} where $0=a_1<a_2 < \cdots < a_n$ and $\gcd(a_2,\dots,a_n)=1$. Every toric ideal of $S_n = \Bbbk[x_1,\dots,x_n]$ defining a curve in $\PP^{n-1}$ is of the form $I_{\ker A}$ for some such matrix $A$, up to reordering columns. In this notation, $\deg I_{\ker A} = a_n$.

The results of \cite{HHS} in conjunction with the general results on curves from \cite{GLP} give a characterization of which monomial curves have maximal regularity, in terms of the entries $\{a_i\}$. We record the following corollary of their results. 

\begin{proposition}\label{prop:toric-curves-iff}
Let $A$ be a matrix as in \eqref{matrixform}, let $d=a_n$, and let $\cL = \ker A$. Then $I_\cL$ has maximal regularity, i.e. $\reg I_\cL = \deg I_\cL - \codim I_\cL + 1$, if and only if one of the following holds:
\begin{itemize}
    \item $d\le n$,
    \item $d \ge n+1$ and the integers $(a_1,a_2,\dots,a_n)$ are either $(0,1,2,\dots,n-3,d-1,d)$ or $(0,1,d-n+3,d-n+4,\dots,d)$.
\end{itemize}
\end{proposition}
\begin{proof}
    Without loss of generality suppose $\Bbbk$ is algebraically closed. By general theory, since $X_\cL$ is toric, $X_\cL$ is rational.
    
    If $d\in \{n-1,n\}$, then \cite{GLP}*{Theorem 3.1} implies that equality is achieved. So assume $d\ge n+1$. Comment 1 at the end of \S2 of \cite{GLP} now says that a necessary condition for $X_\cL$ to achieve equality in EG is that $X_\cL$ be smooth. So henceforth assume that $X_\cL$ is smooth, i.e., that $a_2=1$ and $a_{n-1}=d-1$.
    
    Using the notation of \cite{HHS}, let $\lambda(X_\cL)$ be the length of the longest gap (i.e., the largest value of $a_k-a_{k-1}-1$) in $\cS=\{a_1,a_2,\cdots,a_{n-1},a_n\}$ and let $\eps = \max\{i \mid [0,i], [d-i,d] \subseteq \cS\}$. Observe that the sum of all gaps is $a_n-a_1-n+1 = d-n+1$, so $\lambda(X_\cL) = d-n+1$ is achieved if and only if there is only a single gap of positive length. We then see from \cite{HHS}*{Theorem 2.7} that 
    \[
    \reg I= \reg (S/I) +1\le \frac{\lambda(X_\cL)-1}{\eps} + 3 \le d-n+3
    \]
    with equality only if $\eps = 1$ and $\lambda(X_\cL) = d-n+1$. Since there can only be a single gap, and $\eps=1$ says that $2$ and $d-2$ cannot both appear in $\mathcal S$, we conclude that the only possibilities for $\mathcal S$ are $\{0,1,2,\dots,n-3,d-1,d\}$ or $\{0,1,d-n+3,d-n+4,\dots,d\}$. Note that these are the same list up to $k\mapsto d-k$.
    
    Finally, we need to know that these cases actually give equality. This is now exactly implied by the statement of \cite{HHS}*{Theorem 3.4} with $\eps = 1$ and $p=d-n+3$.
\end{proof}

In the case of codimension $2$, we find the following corollary.

\begin{corollary}\label{cor:toric-n=4-iff}
Suppose $\cL\subseteq\ZZ^4$ is saturated and has rank $2$, and $I_\cL$ is nondegenerate. Then $\reg I_\cL=\deg I_\cL-1$ if and only if $\cL$ satisfies one of the following cases, up to a reordering of coordinates:
\begin{itemize}
    \item $\cL$ has a Gale diagram equal to $\{(0,2),(-1,0),(-1,-1),(2,-1)\}$. In this case, $I_\cL$ is a complete intersection.
    \item $\cL$ has a Gale diagram equal to $\{(1,0),(-1,1),(-1,-d+1),(1,d-2)\}$ for some $d\ge 3$. In this case, $I_\cL$ is not Cohen--Macaulay if $d\ge 4$, and is Cohen--Macaulay but not a complete intersection if $d=3$.
\end{itemize}
\end{corollary}

\begin{proof}
The set of possible Gale diagrams for $\cL$ follows directly from \cref{prop:toric-curves-iff}. The remaining assertions follow from \cref{lem:ci-imbalanced} and \cref{prop:non-cm-open-quadrants}.
\end{proof}

\section{Cohen--Macaulay lattice ideals}
\label{sec:CM}

In \cref{subsec:CM-char-0}, we characterize when $\reg I_\cL=\deg I_\cL-\codim I_\cL+1$ for Cohen--Macaulay lattice ideals $I_\cL\subseteq\langle x_1,\dots,x_n\rangle^2$, under the assumption $\operatorname{char}\Bbbk=0$. Then in \cref{subsec:ci}, we briefly address the case when $I_\cL$ is a complete intersection (in arbitrary characteristic). Finally, in \cref{subsec:codim-2-CMnonCI}, we characterize in arbitrary characteristic the lattice ideals $I_\cL$ of codimension $2$ that are Cohen--Macaulay but not complete intersections that satisfy $\reg I_\cL=\deg I_\cL-\codim I_\cL+1$.

\subsection{The characteristic zero case}\label{subsec:CM-char-0}
Let $\Bbbk$ be a field of characteristic $0$. In Section 3 of \cite{PSti}, the authors prove that $\reg I\le\deg I-\codim I+1$ for any Cohen--Macaulay homogeneous ideal $I\subseteq \langle x_1,\dots,x_n\rangle^2$. Following their proof gives the first part of the following result.

\begin{proposition}\label{prop:general-CM-iff}
Suppose $\operatorname{char}\Bbbk=0$. If $I\subseteq \langle x_1,\dots,x_n\rangle^2$ is a Cohen--Macaulay homogeneous ideal with codimension $c$ and $n_2$ minimal homogeneous generators of degree $2$, then $n_2\le\binom{c+1}{2}$. Further, $\reg I=\deg I-c+1$ if and only if $n_2=\binom{c+1}{2}-1$ or $n_2=\binom{c+1}{2}$, and in this case, the Hilbert series of $S/I$ is given by
\[\frac{1+ct+\sum_{i=2}^{\reg I-1}t^i}{(1-t)^{n-c}}.\]
\noindent If $n_2=\binom{c+1}{2}$, then $\reg I=2$ and all minimal generators of $I$ have degree $2$. Moreover, the minimal free resolution of $I$ has the following form:
\[0\lra S(-c-1)^{c\binom{c+1}{c+1}}\lra\dots\lra S(-3)^{2\binom{c+1}{3}}\lra S(-2)^{1\binom{c+1}{2}}\lra I\lra 0.\]

If $n_2=\binom{c+1}{2}-1$, then $\reg I\ge 3$ and $I$ satisfies one of the following two conditions:
\begin{itemize}
    \item All minimal generators of $I$ have degree $2$, and $\reg I=3$.
    \item $I$ has exactly one minimal generator of degree greater than or equal to $3$, which has degree exactly $\reg I$.
\end{itemize}
\end{proposition}

\begin{proof}
Using the notation in \cite{PSti}, let $N$ be the generic initial ideal of $I$ with respect to the reverse lexicographic order. Then $N\subseteq \langle x_1,\dots,x_n\rangle^2$ is a Cohen--Macaulay Borel ideal, and none of the variables $x_{c+1},\dots,x_n$ appear in the minimal monomial generators of $N$. Furthermore, $I$ and $N$ have the same Hilbert function, so it follows that $n_2\le\binom{c+1}{2}$. Moreover, $x_c^{\reg I}$ is a minimal monomial generator of $N$, and $\overline{N}\coloneqq N\otimes\overline{S}$ is an artinian ideal in $\overline{S}\coloneqq S/(x_{c+1},\dots,x_n)$, where $\deg I$ equals the length of $\overline{S}/\overline{N}$. It then follows that the length of $\overline{S}/\overline{N}$ equals $\reg I+c-1$ if and only if all the monomials of degree $2$ supported on $x_1,\dots,x_c$ lie in $\overline{N}$, except possibly for $x_c^2$. Since $N$ is Borel, this is equivalent to requiring $n_2=\binom{c+1}{2}-1$ or $n_2=\binom{c+1}{2}$. We then see that $\reg I\ge 3$ if $n_2=\binom{c+1}{2}-1$ and $\reg I=2$ if $n_2=\binom{c+1}{2}$.

The Hilbert series of $S/I$ equals the Hilbert series of $S/N$. If $\reg I=\deg I-c+1$, it is straightforward to compute the Hilbert series of $S/N$ using the above information.

Now, suppose $n_2=\binom{c+1}{2}$. Since $\reg I=2$, it follows that $I$ has no minimal generators of degree more than $2$. Moreover, the minimal free resolution of $I$ is pure, and we can see that the resolution takes the desired form (say, by considering the Hilbert series of $S/I$).

Suppose $n_2=\binom{c+1}{2}-1$. Since the Betti numbers of $N$ are at least those of $I$, we find that $I$ has at most one minimal generator of degree at least $3$, and if such a generator exists, it must have degree exactly $\reg I$. If no such generator exists, then $I$ is generated in degree $2$; if furthermore $\reg I\ge 4$, then $N$ has no generator in degree $3$, which leads to a contradiction by Green's crystallization principle \cite{Green}*{Proposition 2.28}. Thus, if $I$ is generated in degree $2$, then $\reg I\le 3$.
\end{proof}

This yields the following corollary for lattice ideals.

\begin{corollary}
Suppose $\operatorname{char}\Bbbk=0$, and that $I_\cL\subseteq\langle x_1,\dots,x_n\rangle^2$ is Cohen--Macaulay, where $\cL\subseteq\ZZ^n$ is a lattice of rank $r$. Then $\reg I_\cL\le\deg I_\cL-r+1$, and the number of fibers of $\cL$ of degree $2$ is at least $\binom{n+1}{2}-\binom{r+1}{2}$. Furthermore, $\reg I_\cL=\deg I_\cL-r+1$ if and only if the number of fibers of $\cL$ of degree $2$ equals $\binom{n+1}{2}-\binom{r+1}{2}$ or $\binom{n+1}{2}-\binom{r+1}{2}+1$.
\end{corollary}

\subsection{Complete intersections}\label{subsec:ci}
If $\operatorname{char}\Bbbk=0$, for complete intersection ideals $I$, \cref{prop:general-CM-iff} states that $\reg I=\deg I-\codim I+1$ if and only if either $\codim I=1$ or $\codim I=2$ and $I$ is generated by two quadratics. But a straightforward computation, using the Koszul complex, allows us to generalize this to arbitrary characteristic. We state this formally below.

\begin{proposition}\label{prop:general-ci-ideals}
Let $\Bbbk$ be a field of arbitrary characteristic. Suppose $I\subseteq\langle x_1,\dots,x_n\rangle^2$ is a homogeneous complete intersection ideal. Then $\reg I\le\deg I-\codim I+1$, with equality holding if and only if $\codim I=1$ or if $\codim I=2$ and $I$ is generated by two quadratics.
\end{proposition}

\begin{proof}
Let $I$ have codimension $m$ and let its minimal generators have degrees $d_1,\dots,d_m\ge 2$. Then $I$ is resolved by a Koszul complex, and one directly finds that $\reg I = d_1 + \cdots + d_m - m +1$ while $\deg I = d_1\cdots d_m$. The given inequality is then equivalent to $d_1+\cdots+d_m \le d_1\cdots d_m$. This is always true when $m=1$ and impossible for $m\ge 3$. When $m=2$, it holds only for $d_1=d_2=2$.
\end{proof}

If $I=I_\cL$ is a nondegenerate complete intersection lattice ideal, we see from \cref{prop:general-ci-ideals} that $\reg I_\cL=\deg I_\cL-\codim I_\cL+1$ only holds if $r\le 2$, and it always holds if $r=1$. If $r=2$, we find the following.

\begin{corollary}\label{cor:finitely-many-ci-lattice-ideals}
Fix a field $\Bbbk$ of arbitrary characteristic. There are only finitely many nondegenerate lattice ideals $I_\cL$ with the following properties:
\begin{itemize}
    \item all Gale vectors are nonzero,
    \item $\codim I_\cL \ge 2$,
    \item $I_\cL$ is a complete intersection, and
    \item $\reg I_\cL = \deg I_\cL - \codim I_\cL + 1$.
\end{itemize}
Moreover, all such lattices have rank $2$, and have $n\le 8$. Up to permutations of the coordinates, there are exactly $23$ such lattice ideals, $14$ of which are saturated (see \cref{ci-lattice-ideals} for a full list).
\end{corollary}

\begin{proof}
Note that any quadratic generator coming from a lattice has at most four variables in its support (it must be of the form $x_px_q-x_rx_s$, for not-necessarily-distinct $p,q,r,s\in\{1,\dots,n\}$). If $I_\cL$ is a complete intersection such that $\reg I_\cL=\deg I_\cL-\codim I_\cL+1$, of codimension at least $2$, \cref{prop:general-ci-ideals} implies that there can be at most $8$ variables from $\{x_1,\dots,x_n\}$ used in its minimal generators, since $I_\cL$ is generated by two quadratics. There are therefore finitely many such lattices. Conducting this finite search explicitly in Macaulay2 \cite{M2}, we obtain the full list in \cref{ci-lattice-ideals}.
\end{proof}

\begin{table}
    \centering
    \begin{tabular}{ccc}
	$n$& Saturated $\cL$ & Non-Saturated $\cL$ \\\hline 
	\addlinespace[1ex]
	$3$ & $\emptyset$ & $\left[\begin{smallmatrix}0&2\\2&-1\\-2&-1\end{smallmatrix}\right]
	\left[\begin{smallmatrix}0&2\\2&-2\\-2&0\end{smallmatrix}\right]$\\\addlinespace[1ex]\hline
	\addlinespace[1ex]
	$4$ &$\left[\begin{smallmatrix}0&2\\2&-1\\-1&0\\-1&-1\end{smallmatrix}\right]$ & $\left[\begin{smallmatrix}0&1\\0&1\\2&-2\\-2&0\end{smallmatrix}\right]
	\left[\begin{smallmatrix}0&1\\0&1\\2&-1\\-2&-1\end{smallmatrix}\right]
	\left[\begin{smallmatrix}0&2\\2&0\\0&-2\\-2&0\end{smallmatrix}\right]
	\left[\begin{smallmatrix}0&2\\2&-1\\0&-1\\-2&0\end{smallmatrix}\right]$\\\addlinespace[1ex]\hline
	\addlinespace[1ex]
	$5$&$\left[\begin{smallmatrix}0&1\\0&1\\2&0\\-1&-1\\-1&-1\end{smallmatrix}\right] 
	\left[\begin{smallmatrix}0&1\\0&1\\2&-1\\-1&0\\-1&-1\end{smallmatrix}\right]
	\left[\begin{smallmatrix}0&1\\0&1\\1&0\\1&-2\\-2&0\end{smallmatrix}\right]
	\left[\begin{smallmatrix}0&2\\1&0\\1&-1\\-1&0\\-1&-1\end{smallmatrix}\right]
	\left[\begin{smallmatrix}0&2\\2&0\\0&-1\\-1&0\\-1&-1\end{smallmatrix}\right]$ & $\left[\begin{smallmatrix}0&1\\0&1\\2&-1\\0&-1\\-2&0\end{smallmatrix}\right]
	\left[\begin{smallmatrix}0&1\\0&1\\2&0\\0&-2\\-2&0\end{smallmatrix}\right]$\\\addlinespace[1ex]\hline
	\addlinespace[1ex]
	$6$&$\left[\begin{smallmatrix}0&1\\0&1\\1&-1\\1&-1\\-1&0\\-1&0\end{smallmatrix}\right]
	\left[\begin{smallmatrix}0&1\\0&1\\1&0\\1&-1\\-1&0\\-1&-1\end{smallmatrix}\right]
	\left[\begin{smallmatrix}0&1\\0&1\\2&0\\0&-1\\-1&0\\-1&-1\end{smallmatrix}\right]
	\left[\begin{smallmatrix}0&1\\0&1\\2&-1\\0&-1\\-1&0\\-1&0\end{smallmatrix}\right]
	\left[\begin{smallmatrix}0&1\\0&1\\2&0\\0&-2\\-1&0\\-1&0\end{smallmatrix}\right]$ &  $\left[\begin{smallmatrix}0&1\\0&1\\2&0\\0&-1\\0&-1\\-2&0\end{smallmatrix}\right]$\\\addlinespace[1ex]\hline
	\addlinespace[1ex]
	$7$&$\left[\begin{smallmatrix}0&1\\0&1\\1&0\\1&-1\\0&-1\\-1&0\\-1&0\end{smallmatrix}\right]
	\left[\begin{smallmatrix}0&1\\0&1\\2&0\\0&-1\\0&-1\\-1&0\\-1&0\end{smallmatrix}\right]$ & $\emptyset$\\\addlinespace[1ex]\hline
	\addlinespace[1ex]
	$8$&$\left[\begin{smallmatrix}0&1\\0&1\\1&0\\1&0\\0&-1\\0&-1\\-1&0\\-1&0\end{smallmatrix}\right]$ & $\emptyset$\\\addlinespace[1ex]
\end{tabular}

    \caption{All lattices (up to a reordering of coordinates) giving a complete intersection with $\reg I_\cL = \deg I_\cL - 1$. Given a matrix as above, the corresponding lattice is the column span.}
    \label{ci-lattice-ideals}
\end{table}

This finite set of lattices is independent of the field $\Bbbk$, by \cite{PS}*{Remark 3.2}.

\subsection{Lattice ideals of codimension $2$}\label{subsec:codim-2-CMnonCI}
In this section, we focus exclusively on Cohen--Macaulay lattice ideals $I_\cL$ of codimension $2$; since the case of complete intersections is addressed in \cref{subsec:ci}, we may assume that the lattice ideals in question are not complete intersections.

First, we set some notation. Given any ideal $I$ of $S$, we write the minimal free resolution of $I$ as follows:
\[0\lra\bigoplus_{j=1}^{b_p}S(-d_{pj})\lra\dots\lra\bigoplus_{j=1}^{b_1}S(-d_{1j})\lra I\lra 0,\]
where $p$ equals the projective dimension of $S/I$. For all $1\le i\le p$, we also let $M_i\coloneqq\max\{d_{ij}\mid 1\le j\le b_i\}$ and $m_i\coloneqq\min\{d_{ij}\mid 1\le j\le b_i\}$. We recall that the resolution is called \textit{pure} if $m_i=M_i$ for all $i$. From \cite{PS}*{Remark 5.8 and Theorem 6.1}, we obtain the following.

\begin{lemma}
\label{lem:cm}
Let $I_\cL\subseteq \Bbbk[x_1,\dots,x_n]$ be a lattice ideal of codimension $2$ that is Cohen--Macaulay but not a complete intersection. Then the minimal free resolution of $I_\cL$ is given by the Hilbert--Burch complex, and satisfies $p=2$, $b_1=3$, $b_2=2$, $d_{11}+d_{12}+d_{13}=d_{21}+d_{22}$, and $M_1<m_2$.
\end{lemma}

We have the following lower bound for the degree of $I_\cL$.

\begin{lemma}[\cite{MNR}*{Theorem 1.2(a), Corollary 1.3}]\label{lem:CMdegbound}
Suppose $I$ is Cohen--Macaulay and $\codim I=2$. Then
$$\deg I\ge \frac{1}{2}m_1m_2+\frac{1}{2}(M_2-M_1)(M_2-m_2+M_1-m_1).$$
Furthermore, the minimal free resolution of $I$ is pure if and only if $\deg I=\frac{1}{2}m_1m_2$ if and only if $\deg I=\frac{1}{2}M_1M_2$.
\end{lemma}

We can use \cref{lem:CMdegbound} to give a necessary and sufficient condition for when a Cohen--Macaulay lattice ideal $I_\cL$ of codimension $2$ that is not a complete intersection satisfies $\reg I_\cL=\deg I_\cL-\codim I_\cL+1=\deg I_\cL - 1$.

\begin{proposition}\label{prop:codim-2-CMnonCI-iff}
Suppose the lattice ideal $I_\cL$ is of codimension $2$ and is Cohen--Macaulay but not a complete intersection, so $I_\cL\subseteq\langle x_1,\dots,x_n\rangle^2$. Then $\reg I_\cL\le\deg I_\cL-1$, with equality if and only if $I_\cL$ is minimally generated by $3$ quadratics. Moreover, if equality holds, then the minimal free resolution of $I_\cL$ has the form
\[0\lra S(-3)^2\lra S(-2)^3\lra I_\cL\lra 0.\]
\end{proposition}

\begin{proof}
The assertion that $I_\cL\subseteq\langle x_1,\dots,x_n\rangle ^2$ follows from \cref{lem:ci-imbalanced}. This implies $m_1\ge 2$. Now, by \cref{lem:CMdegbound},
\begin{align*}
    2\deg I_\cL -2M_2&\ge m_1m_2 + (M_2-M_1)(M_2-m_2+M_1-m_1) - 2M_2\\
    &= m_2(m_1-2) + (M_2 - M_1 - 2)(M_2 - m_2) + (M_2 - M_1)(M_1-m_1).
\end{align*}
By \cref{lem:cm}, $M_1 < m_2 \le M_2$, so $\reg I_\cL=M_2-1$. In addition, the quantity $(M_2-M_1-2)(M_2-m_2)$ is always nonnegative; this is clear if $M_2 - M_1\ge 2$, and if $M_2-M_1 = 1$ then necessarily $m_2 = M_2$, since $$M_1 < m_2 \le M_2 = M_1 + 1.$$
It follows that $\reg I_\cL=M_2-1\le\deg I_\cL-1$, with equality holding only if $m_1=2 = M_1$, which happens if and only if all minimal generators are of degree $2$. By \cref{lem:cm}, $I_\cL$ has exactly $b_1=3$ minimal generators.

Conversely, if all three minimal generators are of degree $2$, then by \cref{lem:cm}, $d_{21}+d_{22}=d_{11}+d_{12}+d_{13}=6$ and $d_{21},d_{22}\ge m_2>M_1=2$, so $d_{21}=d_{22}=3$. Thus, the minimal free resolution of $I_\cL$ is pure. Examining the degrees of the twists, we find $\reg I_\cL=2$, and by \cref{lem:CMdegbound}, we have $\deg I_\cL=3$, so equality holds. The claimed Betti numbers for the minimal free resolution of $I_\cL$ follow as well.
\end{proof}

In fact, using the results in \cite{PS} (or by mimicking the argument of \cref{cor:finitely-many-ci-lattice-ideals}), it can be seen that the requirement given in \cref{prop:codim-2-CMnonCI-iff} that $I_\cL$ be minimally generated by $3$ quadratics shows that there can only be finitely many rank $2$ lattices $\cL$ with no zero Gale vectors such that $I_\cL$ is nondegenerate, Cohen--Macaulay, not a complete intersection, and satisfies $\reg I_\cL=\deg I_\cL-1$. After performing this finite search using \cite{M2}, we find the following.

\begin{proposition}\label{prop:codim-2-CMnonCI-list}
Suppose the lattice $\cL$ has rank $2$, and that $I_\cL$ is Cohen--Macaulay but not a complete intersection, which implies $I_\cL\subseteq\langle x_1,\dots,x_n\rangle^2$. Then $\reg I_\cL\le\deg I_\cL-1$, with equality if and only if $\cL$ has a Gale diagram in which the subset of nonzero vectors equals one of the following:
\begin{itemize}
    \item $\{(1,1),(1,-2),(-2,1)\}$,
    \item $\{(1,0),(0,1),(1,-2),(-2,1)\}$,
    \item $\{(1,1),(0,1),(-1,0),(-1,-1),(1,-1)\}$,
    \item $\{(1,0),(0,1),(1,1),(-1,0),(0,-1),(-1,-1)\}$.
\end{itemize}
Of these four possibilities for $\cL$, the first is not saturated, and the remaining possibilities are saturated.
\end{proposition}

By combining \cref{prop:codim-2-CMnonCI-list} with \cref{cor:finitely-many-ci-lattice-ideals}, we see that any nondegenerate Cohen--Macaulay lattice ideal $I_\cL$ of codimension $2$ satisfies $\reg I_\cL=\deg I_\cL-1$ if and only if its subset of nonzero Gale vectors is given by one of the finitely many options from \cref{prop:codim-2-CMnonCI-list} and \cref{ci-lattice-ideals}.

\begin{remark}\label{rmk:codim-3-gorenstein}
The setting of \cite{MNR} also allows us to determine when $\reg I=\deg I-\codim I+1$ for nondegenerate Gorenstein ideals $I$ of codimension $3$: from \cite{MNR}*{Theorem 1.4, Corollary 1.5}, we see that $\reg I\le\deg I-2$ for all nondegenerate Gorenstein ideals $I$ of codimension $3$, with equality holding if and only if the minimal free resolution of $I$ has the form
\[0\lra S(-5)\lra S(-3)^5\lra S(-2)^5\lra I\lra 0.\]
In this case, $\reg I=3$ and $\deg I=5$. This can be rephrased to say that equality is achieved if and only if all the minimal generators of $I$ have degree $2$ and $I$ is not a complete intersection.
\end{remark}

\section{Reduction to curves in $\PP^3$}
\label{sec:PSrecall}
The remainder of the paper is devoted to characterizing the non-Cohen--Macaulay toric ideals of codimension $2$ that have maximal regularity. By \cref{prop:non-cm-open-quadrants}, all such toric ideals have at least $4$ nonzero Gale vectors, and \cref{cor:toric-n=4-iff} finishes the case when $\cL$ has exactly $4$ nonzero Gale vectors. Thus, we are able to restrict to the case when $\cL$ has at least $5$ nonzero Gale vectors; see \cref{prop:non-cm-main-result}. We then combine \cref{prop:non-cm-main-result} with previous results to complete the proof of \cref{thm:main}.

In this section, we recall the tools that were used in \cite{PS} to prove \cref{thm:reduction}, restricting our attention to the case when $I_\cL$ has codimension $2$, equivalently, when $\cL$ has rank $2$. The strategy of \cite{PS} is to reduce to the case of curves in $\PP^3$.\\

As usual, we assume that $\cL\subseteq\ZZ^n$ is orthogonal to the all-$1$'s vector $(1,1,\dots,1)\in\ZZ^n$. If $I_\cL$ is not Cohen--Macaulay, then the projective dimension of the $S$-module $S/I_\cL$ equals $3$. We fix a Gale diagram $G$ for $\cL$. Much of our work is motivated by the following result.
\begin{theorem}[\cite{PS}*{Theorem 7.3 and Proposition 7.7}]
\label{thm:reduction}
Let $I_\cL\subseteq \Bbbk[x_1,\dots,x_n]$ be a lattice ideal of codimension $2$ that is not Cohen--Macaulay. Then $\reg I_\cL \le \deg I_\cL$, and this inequality is strict if $I_\cL$ is toric. Furthermore, if equality holds, then any Gale diagram for $\cL$ lies on two lines through the origin in $\RR^2$.
\end{theorem}

Given $C\in\Gamma=\ZZ^n/\cL$ and $\ba\in C$, the monomials of degree $C$ are in bijection with the lattice points in $P_\ba\coloneqq\operatorname{conv}(\{\bu\in\ZZ^2\mid B\bu\le\ba\})$. Note that $P_\ba$ and $P_{\ba'}$ are lattice translates if $\ba-\ba'\in \cL$, so by considering polygons up to translation, we define $P_C\coloneqq P_\ba$ for $\ba\in C$. We say $P_C$ is \textit{primitive} if it contains no lattice points other than its vertices. For each $C\in\Gamma$, let $\Delta_C$ be the simplicial complex on $\{1,\dots,n\}$ generated by the supports of all monomials in the fiber of $C$.

 If $S/I_\cL$ has a minimal $i$th syzygy of degree $C$, then by Theorem 3.4 of \cite{PS}, the simplicial complex $\Delta_C$ is homologous to the $(i-1)$-sphere, and $P_C$ is primitive. If $i=3$, then $P_C$ is a primitive parallelogram and is called a \textit{syzygy quadrangle}. Furthermore, in this case, $\Delta_C$ has the homology of the $2$-sphere. The minimal free resolution of $S/I_\cL$ is controlled by the syzygy quadrangles of $I_\cL$, in a way that is made precise in \cite{PS}. 

Given $\bv,\bw\in\ZZ^2$ with $\abs{\det(\bv,\bw)}=1$, let $[\bv,\bw]\coloneqq\operatorname{conv}(\{(0,0),\bv,\bw,\bv+\bw\})$ be a primitive parallelogram. Recall that a vector $\ba \in [\bv,\bw]$ is \textit{supported by} a vector $\bb$ if $\bb \cdot \ba > \bb\cdot \ba'$ for all $\ba'\in [\bv,\bw]\setminus \{\ba\}$. 
\begin{proposition}[\cite{PS}*{Corollary 4.2}]\label{prop:PS-syz-quads}
Suppose $I_\cL$ is not a complete intersection. Then the parallelogram $[\bv,\bw]$ is a syzygy quadrangle if and only if each vertex of $[\bv,\bw]$ is supported by at least one vector in the Gale diagram $G$.
\end{proposition}
In particular, if $I_\cL$ is not a complete intersection, then by \cref{prop:non-cm-open-quadrants}, $I_\cL$ is not Cohen--Macaulay if and only if there exists a Gale diagram for which the unit square $[(1,0),(0,1)]$ is a syzygy quadrangle, or in other words, if and only if $I_\cL$ has at least one syzygy quadrangle (in an arbitrary Gale diagram).

For convenience, we record the following two results, which are implicit in \cite{PS}.

\begin{lemma}\label{lem:contractible-complexes}
Suppose the polygon $P\subseteq\RR^2$ equals $P_{\ba_0}$ for some $\ba_0\in\ZZ_{\ge 0}^n$. Then $P$ contains $0$. Moreover, let $\mathcal{S}$ be the set of all $C\in\Gamma$ such that $P_C$ equals $P$ up to lattice translations, and let $C_0$ be the unique member of $\Gamma$ that contains $\bv=(v_1,\dots,v_n)\in\ZZ^n$, where $v_i$ equals the maximum of $\bb_i\cdot\bu$ for $\bu\in P$. Then $C_0\in\mathcal{S}$, and $\Delta_C$ is contractible for all $C\in\mathcal{S}\setminus\{C_0\}$.
\end{lemma}

\begin{proof}
Note $0\in P$ since $\ba_0\in\ZZ_{\ge 0}^n$. Also, $P_\bv=P$, so $P_{C_0}$ equals $P$ up to lattice translations, and $C_0\in\mathcal{S}$. Let $C\in\mathcal{S}\setminus\{C_0\}$, so there exists $\ba=(a_1,\dots,a_n)\in C$ such that $P_\ba=P$. Note that $\bv\le\ba$. Since $C\neq C_0$, we have $\bv\neq\ba$. Thus, $v_i\le a_i-1$ for some $1\le i\le n$, so $a_i-\bb_i\cdot\bu\ge 1$ for all $\bu\in P$. This means all monomials in the fiber of $C$ are a multiple of $x_i$. Thus, each of the generators of the simplicial complex $\Delta_C$ contains $i\in\{1,\dots,n\}$, so $\Delta_C$ is contractible, as desired.
\end{proof}

\begin{corollary}\label{cor:syzygy-degrees}
Suppose that $C\in\Gamma$ is the degree of a minimal $i$th syzygy of $S/I_\cL$, where $i=1,2,3$. Thus, $C\cap\ZZ_{\ge 0}^n$ is nonempty, and contains some $\ba_0$. Define $\ba\in\ZZ^n$ by letting its $j$th coordinate equal $\max\{\bb_j\cdot\bu\mid\bu\in P_{\ba_0}\}$ for all $1\le j\le n$. Then $\ba\in C$.
\end{corollary}
One particular consequence of this result is that if $C_1,C_2$ are each degrees of minimal syzygies, and $P_{C_1}=P_{C_2}$ (up to translations), then $C_1=C_2$.
\begin{proof}[Proof of \cref{cor:syzygy-degrees}]
By our assumption on $C$, it follows that $\Delta_C$ has the homology of the $(i-1)$-sphere, and is therefore non-contractible. The result then follows from \cref{lem:contractible-complexes}.
\end{proof}

A crucial insight in the proof of \cref{thm:reduction} is that one can reduce to the case of a curve in $\PP^3$. We briefly outline the argument from \cite{PS} here, and we fix the following notation for the remainder of the paper. 

\begin{definition}
\label{reductionprocess}
A \textit{reduction datum} is a triple $(\cL,G,Q)$ with the following properties:
\begin{itemize}
    \item $\cL \subset \ZZ^n$ is a lattice for which $I_\cL$ is not Cohen--Macaulay,
    \item $G$ is a Gale diagram for $\cL$ for which the unit square is a syzygy quadrangle, and
    \item $Q = (Q_1,Q_2,Q_3,Q_4)$ is a partition of $G$ (that is, $G=\bigsqcup_{i=1}^4 Q_i$) such that every vector in $Q_i$ lies in the $i$th closed quadrant.
\end{itemize}
Given such a reduction datum, we define the $\Bbbk$-algebra morphism $$\phi_Q \colon \Bbbk[x_1,\dots,x_n] \to \Bbbk[y_1,\dots,y_4]$$ so that $\phi_Q(x_i) = y_j$ if and only if $\bb_j\in Q_i$. Now define $$J_Q \coloneqq \phi_Q(I_\cL), \qquad I_Q \coloneqq (J_Q \colon (y_1y_2y_3y_4)^\infty)$$ to be ideals of $\Bbbk[y_1,\dots,y_4]$. The ideal $I_Q$ is a lattice ideal corresponding to a lattice $\cL_Q\subseteq \ZZ^4$. The lattice $\cL_Q$ has a canonical Gale diagram $G_Q\coloneqq\{\bb_1',\dots,\bb_4'\}$, where $\bb_i' \coloneqq \sum_{\bb \in Q_i} \bb$.
\end{definition}

\begin{remark} Several properties of the above objects should be noted.
\begin{enumerate}
\renewcommand{\theenumi}{\alph{enumi}}
    \item The map $\phi_Q$ is surjective since $G$ contains a vector in each open quadrant. This ensures that $J_Q$ is indeed an ideal.
    \item We have $\codim J_Q = 2$, so $J_Q$ defines a (possibly reducible, nonreduced) curve in $\PP^3$.
    \item The ideal $J_Q$ need not be a lattice ideal. However, it \textit{is} a lattice ideal after saturation with respect to $y_1y_2y_3y_4$ (hence the existence of $\cL_Q$ as above).
    \item It is not hard to show that $\deg J_Q = \deg I_\cL$. By general properties of saturation, it follows that $\deg I_Q \le \deg J_Q = \deg I_\cL.$ 
\end{enumerate}
\end{remark}

The following result allows us to control the regularity of $I_Q$.

\begin{lemma}\label{lem:preserve-syz-quad-and-degree}
Suppose that $[\bv,\bw]$ is a syzygy quadrangle for $I_\cL$, and that
\begin{align*}
    Q_1\subseteq\{\bu\in\RR^2\mid\text{$\bu\cdot\bv\ge 0$ and $\bu\cdot\bw\ge 0$}\},\qquad Q_2\subseteq\{\bu\in\RR^2\mid\text{$\bu\cdot\bv\le 0$ and $\bu\cdot\bw\ge 0$}\},\\
    Q_3\subseteq\{\bu\in\RR^2\mid\text{$\bu\cdot\bv\le 0$ and $\bu\cdot\bw\le 0$}\},\qquad Q_4\subseteq\{\bu\in\RR^2\mid\text{$\bu\cdot\bv\ge 0$ and $\bu\cdot\bw\le 0$}\}.
\end{align*}
Then, $[\bv,\bw]$ is a syzygy quadrangle for $I_Q$. Furthermore, the total degrees of the corresponding degrees in $\Gamma=\ZZ^n/\cL$ and $\ZZ^4/\cL_Q$ are equal.
\end{lemma}

\begin{proof}
The fact that $[\bv,\bw]$ is a syzygy quadrangle for $I_{Q}$ follows directly from \cref{prop:PS-syz-quads}. Now, let $C\in\Gamma$ be such that $P_C=[\bv,\bw]$ and $I_\cL$ has a minimal third syzygy in degree $C$, and let $C'\in\ZZ^4/\cL_Q$ be such that $P_{C'}=[\bv,\bw]$ and $I_{Q}$ has a minimal third syzygy in degree $C'$. By \cref{cor:syzygy-degrees},
\[\deg C=\sum_{j=1}^n\max\{\bb_j\cdot\bu\mid\bu\in[\bv,\bw]\},\qquad\deg C'=\sum_{i=1}^4\max\{\bb_i'\cdot\bu\mid\bu\in[\bv,\bw]\}.\]
It suffices to show that
\[\max\{\bb_i'\cdot\bu\mid\bu\in[\bv,\bw]\}=\sum_{j\in Q_i}\max\{\bb_j\cdot\bu\mid\bu\in[\bv,\bw]\}\]
for all $1\le i\le 4$, but this follows from the assumptions on the $Q_i$.
\end{proof}

Fix a reduction datum $(\cL, G, Q)$ for which the unit square is a syzygy quadrangle of $I_\cL$ attaining the regularity. It follows from \cref{lem:preserve-syz-quad-and-degree} that the unit square is a syzygy quadrangle for $I_{Q}$, and the monomials corresponding to its vertices retain the same total degree. Thus, $\reg I_{Q} \ge \reg I_{\cL}$.

Similarly, it follows from \cref{cor:syzygy-degrees} that if $[\bv,\bw]$ is a syzygy quadrangle of both $I_\cL$ and $I_{Q}$, then its total degree with respect to $I_{Q}$ is at most its total degree with respect to $I_\cL$. Thus, the regularity can only \textit{strictly} increase if the $\cL\to \cL_Q$ reduction process introduces a new syzygy quadrangle.

In any case, since $\deg I_{Q}\le \deg I_\cL$ and $\reg I_\cL\le\reg I_{Q}$, if \cref{thm:reduction} holds for $n=4$, then \textit{all} of the following three inequalities hold:
\begin{subequations}
\label{reductionchain}
\begin{align}
    \reg I_\cL &\le \reg I_{Q}, \label{ineqA}\\ 
    \reg I_{Q}&\le \deg I_{Q},\label{ineqB} \\
    \deg I_{Q} &\le \deg I_\cL.\label{ineqC}
\end{align}
\end{subequations}

\noindent Therefore $\reg I_\cL \le \deg I_\cL$ and the result holds for general $n$. This is the strategy adopted by \cite{PS} to prove \cref{thm:reduction}.

By inspecting this reduction process and the inequality chain \eqref{reductionchain} in more detail, we are able to classify the possible Gale diagrams of lattice ideals $I_\cL$ satisfying $\reg I_\cL = \deg I_\cL - 1$. Observe that if $\reg I_\cL = \deg I_\cL - 1$, then inequality must occur in exactly one of \eqref{ineqA}, \eqref{ineqB}, or \eqref{ineqC}. In the next section, we will analyze the various possibilities for where equality occurs in \eqref{reductionchain}.

\section{Degree and regularity after reduction to a curve}
\label{sec:analyzingreduction}

In this section we carry out a careful analysis of the inequality chain \cref{reductionchain}. This gives the necessary tools to determine when $\reg I_\cL = \deg I_\cL - 1$ for a toric ideal of codimension 2.

\subsection{Maximal regularity implies equality in \eqref{ineqA}}
\label{sec:maxregeqA}

In this subsection, we use \cref{prop:PS-syz-quads} to compare the syzygy quadrangles of $I_Q$ to those of $I_\cL$ in the case when $\deg I_\cL=\deg I_Q$, which then lets us establish that equality holds in \eqref{ineqA} if $I_\cL$ is a toric ideal with maximal regularity.

Fix an arbitrary reduction datum $(\cL, G, Q)$. 

\begin{lemma}\label{lem:J-contains-line}
Let $J_Q$ be as in \cref{reductionprocess}. The following two statements hold.
\begin{enumerate}
    \item Any associated prime $\frp$ of $J_Q$ of codimension $2$ that contains $y_1y_2y_3y_4$ is of the form $\frp=\langle y_i,y_j\rangle$ for some distinct $i,j\in\{1,2,3,4\}$.
    \item Let $i,j\in\{1,2,3,4\}$ be distinct. Then $\langle y_i,y_j\rangle$ is an associated prime of $J_Q$ if and only if for all nonzero $\bu\in\ZZ^2$, there is some $\bb\in Q_i\cup Q_j$ such that $\bb\cdot\bu<0$. In particular, this can hold only if $\{i,j\}=\{1,3\}$ or $\{2,4\}$.
\end{enumerate}
\end{lemma}

\begin{proof}
Statement (1) follows from Corollary 2.1 of \cite{HS} (though $J_Q$ is not necessarily a lattice basis ideal, we only need that the saturation of $J_Q$ with respect to $y_1y_2y_3y_4$ is a lattice ideal to apply this result). 

Statement (2) follows since $\langle y_i,y_j\rangle$ is an associated prime of $J_Q$ if and only if $\langle y_i,y_j\rangle\supseteq J_Q$, since $\codim J_Q = 2$. Translating this into a condition based on the Gale diagram gives the result.

\end{proof}

\begin{lemma}\label{lem:halfspaces}
We have that $\deg I_\cL = \deg I_{Q}$ if and only if there are two nonzero vectors $\bu_{13},\bu_{24}$ for which $$Q_1\cup Q_3 \subseteq \{\bv \in \RR^2 \mid \bv \cdot \bu_{13} \ge 0\}, \qquad Q_2\cup Q_4 \subseteq \{\bv\in \RR^2 \mid \bv \cdot \bu_{24} \ge 0\}.$$
\end{lemma}
\begin{proof}
Recall that $I_{Q}$ is a lattice ideal that equals the saturation $(J_Q:(y_1y_2y_3y_4)^\infty)$. It follows that $\deg I_{Q} = \deg J_Q$ if and only if no codimension-$2$ associated prime of $J_Q$ contains $\langle y_1y_2y_3y_4\rangle$, so by \cref{lem:J-contains-line}, we see that $\deg I_{Q}=\deg J_Q$ if and only if $J_Q$ does not have an associated prime of the form $\langle y_i,y_j\rangle$ for distinct $i,j\in\{1,2,3,4\}$. The desired result then follows from \cref{lem:J-contains-line} and the fact that $\deg J_Q = \deg I_\cL$.
\end{proof}

\begin{lemma}
\label{lem:eqdeg}
If $\deg I_\cL=\deg I_{Q}$, then every syzygy quadrangle of $I_{Q}$ (with respect to $G_{Q}$) is a syzygy quadrangle of $I_\cL$ (with respect to $G$).
\end{lemma}

\begin{proof}
Suppose that $I_{Q}$ has a syzygy quadrangle that is not a syzygy quadrangle for $I_\cL$. Without loss of generality, after translating we may assume that this syzygy quadrangle equals $[\bv,\bw]$ for some vectors $\bv,\bw$ with nonnegative $y$-coordinates. Since $\abs{\det(\bv,\bw)}=1$, after translating again if necessary, we may assume that $\bv,\bw$ both lie in the first closed quadrant or both lie in the second closed quadrant. Suppose without loss of generality that $\det(\bv,\bw)=1$.

Assume that $\bv,\bw$ both lie in the first closed quadrant. The case where $\bv,\bw$ both lie in the second closed quadrant is similar. By \cref{prop:PS-syz-quads}, each vertex of the syzygy quadrangle is supported by some vector in $G_{Q}$. Thus there exist vectors $\bc_1',\bc_2',\bc_3',\bc_4'\in G_{Q}$ such that $\bc_1'\cdot\bv,-\bc_2'\cdot\bv,-\bc_3'\cdot\bv,\bc_4'\cdot\bv>0$ and $\bc_1'\cdot\bw,\bc_2'\cdot\bw,-\bc_3'\cdot\bw,-\bc_4'\cdot\bw>0$. Since $G_{Q}$ contains exactly four vectors, one in each open quadrant, we necessarily have that $\bc_i'$ is the unique vector of $G_{Q}$ that lies in the $i$th open quadrant, that is, $\bc_i'=\bb_i'$. Since $[\bv,\bw]$ is not a syzygy quadrangle for $I_\cL$, the Gale diagram $G$ does not contain an analogous collection of four vectors $\bc_1,\bc_2,\bc_3,\bc_4$ satisfying the same inequalities as above. Since $G$ contains a vector in each open quadrant, we see that there exist $\bc_1,\bc_3$ satisfying the conditions. Thus, there either does not exist $\bc_2\in G$ such that $-\bc_2\cdot\bv,\bc_2\cdot\bw>0$, or there does not exist $\bc_4\in G$ such that $\bc_4\cdot\bv,-\bc_4\cdot\bw>0$. Assume the first case holds; the second case is similar. This implies that each $\bb\in Q_2$ either satisfies $\bb\cdot\bv\ge 0$ or $\bb\cdot\bw\le 0$. Since $\sum_{\bb\in Q_2}\bb=\bb_2'$, we see that there must exist some nonzero $\bc\in Q_2$ such that $\bc\cdot\bv\ge 0$, and there exists some nonzero $\bd\in Q_2$ such that $\bd\cdot\bw\le 0$. But these conditions prevent the nonzero vectors $\bc,\bd,\bb_4'$ from all lying in a single closed half-plane with boundary passing through the origin, contradicting \cref{lem:halfspaces}, which asserts the existence of some nonzero $\bu\in\ZZ^2$ such that $0\le\bu\cdot\bc,\bu\cdot\bd$ and $0\le\bu\cdot\sum_{\bb\in Q_4}\bb=\bu\cdot\bb_4'$.
\end{proof}

We can now prove our first key result.

\begin{corollary}
\label{cor:eqreg}
If $\reg I_\cL \ge \deg I_\cL - 1$, then $\reg I_{Q} = \reg I_\cL$.
\end{corollary}
\begin{proof}
If $\reg I_{Q} > \reg I_\cL$, then as stated in \cref{sec:PSrecall}, there is a syzygy quadrangle of $I_{Q}$ that is not a syzygy quadrangle of $I_\cL$. Then \cref{lem:eqdeg} implies that $\deg I_{Q} < \deg I_\cL$. But then $\reg I_\cL \le \deg I_{\cL} - 2$, a contradiction.
\end{proof}

\subsection{Existence of a reduction achieving equality in \eqref{ineqB}}
\label{sec:reg=degreduction}

The following lemma allows us to prove our next result, \cref{prop:regdeg}, which says that if $\reg I_\cL=\deg I_\cL-1$ and certain nonrestrictive technical conditions are satisfied, then there exists some partition $Q$ for which $\reg I_Q=\deg I_Q$.

\begin{lemma}
\label{lem:galediagrams}
Suppose that $\cL\subseteq \ZZ^4$ with $\reg I_\cL = \deg I_\cL - 1$, and let $G$ be a Gale diagram for $\cL$ such that the unit square is a syzygy quadrangle attaining the regularity. Then $G$ either lies on two lines or has the following form, up to dihedral symmetries: $$G = \{(1,a),(-1,d-1),(-1,1-a),(1,-d)\}$$ for some $a,d > 1$.
\end{lemma}
\begin{proof}
Let $C\in \ZZ^4/\cL$ be the multidegree for which $P_C$ is the unit square. Then $\reg I_\cL = \deg C - 2$ since syzygy quadrangles correspond to third syzygies \cite{PS}*{Theorem 3.4}. We start in the same spirit as \cite{PS}, Proposition 7.7. Suppose the Gale vectors are given by $$\ba = (a_1,a_2), \quad \bb = (-b_1,b_2), \quad \bc = (-c_1,-c_2), \quad \bd=(d_1,-d_2),$$
where $a_i,b_i,c_i,d_i > 0$ for all $i=1,2$. By rotating $180^\circ$ if necessary, we can without loss of generality assume that $b_2\le d_2$.

We now have the following chain of comparisons:
\begin{align*}
    \deg I_\cL - 1 = \reg I_\cL &= \deg C - 2\\
    &= a_1+b_2+a_2+d_1-2\\
    &\le (a_1 + d_2-1) + (a_2 + d_1-1) \tag{1}\\
    &\le a_1d_2 + a_2d_1\tag{2}\\
    &= \abs{\det(\bd,\ba)}\\
    &\le \deg I_\cL. \tag{3}
\end{align*}
There are three inequalities, and hence three possibilities for where the jump by 1 can occur. We treat these case by case.\\

\noindent \textit{Case 1: }Inequality occurs at (1).
In this case we have the following: $$d_2 = b_2+1, \quad (a_1-1)(d_2-1)=(a_2-1)(d_1-1)=0, \quad \abs{\det(\bd,\ba)} = \deg I_\cL.$$ Since $d_2=b_2+1$, we also have $c_2=a_2-1$.
If $a_2=1$ then $c_2=0$, so we conclude that $d_1=1$ and $a_2 > 1$. Similarly if $d_2=1$ then $b_2=0$ so we conclude that $a_1=1$ and $b_2 > 1$. Since $b_1+c_1=a_1+d_1=2$, we conclude that $b_1=c_1=1$. Thus the Gale vectors have the form $$\ba = (1,a_2), \quad \bb = (-1,d_2-1), \quad\bc= (-1,1-a_2), \quad \bd = (1,-d_2).$$ 

\noindent \textit{Case 2: }Inequality occurs at (2).
In this case we have the following: $$d_2 = b_2, \quad (a_1-1)(d_2-1)+(a_2-1)(d_1-1)=1, \quad \abs{\det(\bd,\ba)} = \deg I_\cL.$$
Now we have $c_2=a_2$. 
\begin{itemize}
    \item \textit{Case 2.1:} $a_1=1$ and $a_2=d_1=2$. Then the Gale vectors have the form $$\ba = (1,2), \quad \bb = (-b_1,d_2), \quad \bc = (-c_1, -2), \quad \bd = (2,-d_2).$$ We then also have $$4+c_1d_2 = \abs{\det(\bc,\bd)} \le \deg I_\cL = \abs{\det(\ba,\bd)} = 4+d_2,$$ forcing $c_1=1$. We conclude that $b_1=2$, so $\bb+\bd=\ba+\bc=0$.
    \item \textit{Case 2.2:} $d_2=1$ and $a_2=d_1=2$. Then the Gale vectors have the form $$\ba = (a_1,2), \quad \bb = (-b_1,1), \quad \bc = (-c_1,-2), \quad \bd = (2,-1).$$ Then we again have $$4+c_1 = \abs{\det(\bc,\bd)} \le \deg I_\cL = \abs{\det(\ba,\bd)} = 4+a_1,$$ so $c_1\le a_1$. By the same inequality with $\abs{\det(\ba,\bb)}$, we obtain $b_1\le 2$. Then $$2+a_1=b_1+c_1\le 2+c_1,$$ so $a_1\le c_1$. We conclude that $a_1=c_1$. So $b_1=2$ and again $\bb+\bd = \ba+\bc$.
    \item \textit{Case 2.3:} $a_1=d_2=2$ and $a_2=1$. Proceed as in Case 2.1 to conclude that $G$ lies on two lines.
    \item \textit{Case 2.4:} $a_1=d_2=2$ and $d_1=1$. Proceed as in Case 2.2 to conclude that $G$ lies on two lines.
\end{itemize}

\noindent \textit{Case 3: }Inequality occurs at (3).
In this case we have the following: $$d_2 = b_2, \quad (a_1-1)(d_2-1)=(a_2-1)(d_1-1)=0, \quad \abs{\det(\bd,\ba)} = \deg I_\cL-1.$$
We have $a_2=c_2$ once again. 
\begin{itemize}
    \item \textit{Case 3.1:} $a_1=a_2=1$. Then the Gale vectors have the form $$\ba = (1,1), \quad \bb = (-b_1,d_2), \quad \bc = (-c_1,-1), \quad \bd = (d_1,-d_2).$$ We now have $$c_1d_1+d_2 = \abs{\det(\bc,\bd)} \le \deg I_\cL = \abs{\det(\bd,\ba)}+1 = d_1+d_2+1.$$ Then either $c_1=1$ or $(c_1,d_1)=(2,1)$. In the former case, $d_1=b_1$ and the Gale diagram lies on two lines. In the latter case, $b_1=2$ the Gale diagram must be $$\ba = (1,1), \quad \bb = (-2,d_2), \quad \bc = (-2,-1), \quad \bd = (1,-d_2).$$ But now we see that $$2d_2+2=\abs{\det(\bb,\bc)} \le \deg I_\cL = \abs{\det(\bd,\ba)}+1 = d_2+2,$$ which is absurd since $d_2 > 0$, so $(c_1,d_1)=(2,1)$ does not occur.
    \item \textit{Case 3.2:} $a_1=d_1=1$. Since $b_1+c_1=a_1+d_1=2$, we conclude that $b_1=c_1=1$. Then the Gale vectors have the form $$\ba = (1,a_2), \quad \bb = (-1,d_2), \quad \bc = (-1,-a_2), \quad \bd = (1,-d_2),$$ and so lie on two lines.
    \item \textit{Case 3.3:} $a_2=d_2=1$. Proceed as in Case 3.2 to conclude that $G$ lies on two lines.
    \item \textit{Case 3.4:} $d_1=d_2=1$. Proceed as in Case 3.1 to conclude that $G$ lies on two lines.
\end{itemize}
\end{proof}

We now prove our second key result. 
\begin{proposition}
\label{prop:regdeg}
Suppose $n\ge 5$, and fix a reduction datum $(\cL, G, Q)$ with the following properties:
\begin{enumerate}
    \item $\reg I_\cL = \deg I_\cL - 1$,
    \item $\reg I_Q = \deg I_Q - 1$,
    \item the unit square is a syzygy quadrangle attaining the regularity of $I_\cL$,
    \item $G$ consists only of nonzero vectors, and
    \item $G$ is not contained in two lines.
\end{enumerate} Then, there is a different choice of partition $R$ of $G$ for which $\reg I_{R}=\deg I_{R}$.
\end{proposition}
\begin{proof}
Recall that the unit square is a syzygy quadrangle of $I_Q$, with the same total degree as that of $I_\cL$. Since $\reg I_Q = \reg I_\cL$ by \cref{cor:eqreg}, it follows that the unit square still attains the regularity. We therefore know that $G_Q = \{\bb_1',\bb_2',\bb_3',\bb_4'\}$ is of one of the forms described in \cref{lem:galediagrams}. Suppose that $G_Q$ lies on two lines $\ell_1 \supseteq \{\bb_1',\bb_3'\}$ and $\ell_2\supseteq \{\bb_2',\bb_4'\}$ passing through the origin. Since $\reg I_Q=\reg I_\cL$, we see that $\deg I_\cL = \deg I_Q$, so \cref{lem:halfspaces} implies that there is a closed half-plane $\cH_{13}$ containing $Q_1\cup Q_3$. Since $\cH_{13}$ is closed under addition, this implies that $\bb_1',\bb_3'\in \cH_{13}$, so we conclude that $\partial \cH_{13}=\ell_1$. But then every vector in $Q_1$ must lie on $\ell_1$, otherwise $\bb_1'\in \cH_{13}\setminus \ell_1$. Likewise, $Q_3\subseteq \ell_1$. Analogously, $Q_2,Q_4\subseteq \ell_2$. This contradicts that $G$ is not contained in two lines. So henceforth suppose $G_Q$ is not contained in two lines.

Let $\ell^+$ and $\ell^-$ denote the positive and negative $y$-axis, respectively. If $G \cap \ell^+$ and $G\cap \ell^-$ are both nonempty, then \cref{lem:halfspaces} implies that there is a partition $R$ with $\deg I_R<\deg I_\cL$ (choose $R$ so that $R_1$ contains vectors from $G\cap \ell^+$ and $R_3$ contains vectors from $G\cap \ell^-$). This would imply that $\reg I_R=\deg I_R$. So henceforth, without loss of generality, assume that $G\cap \ell^-=\emptyset$.

\cref{lem:galediagrams} implies that up to reflections over the axes, we have $$G_Q = \{(1,a),(-1,d-1),(-1,1-a),(1,-d)\}$$ for some $a,d > 1$. Combining this with $G\cap \ell^-=\emptyset$, we have $$G = \{(1,a'),(-1,b'),(-1,1-a),(1,-d)\} \cup Q_1^{x=0}\cup Q_2^{x=0},$$ where $Q_i^{x=0}$ is the intersection of $Q_i$ with the $y$-axis, and $0 < a' \le a$ and $0 < b' \le d-1$. Suppose that $Q_2^{x=0}$ is nonempty, that is, there is some $(0,\eps)\in Q_2$. Since $\deg I_\cL = \deg I_Q$, we can choose a half-plane $\cH_{24} = \{\bv\in \RR^2 \mid \bv \cdot \bu \ge 0\}$ containing $Q_2\cup Q_4$. If $\bu = (u_1,u_2)$, then $$-u_1 + u_2b' \ge 0, \quad u_1-u_2d \ge 0, \quad u_2\ge 0,$$ which quickly implies $u_1 = u_2 = 0$, a contradiction. So $Q_2^{x=0}=\emptyset$ and $b'=d-1$.
If $a' < a-1$, then the same argument shows that $Q_1^{x=0}=\emptyset$. But this is impossible, since $n\ge 5$. So $a'\ge a-1$. Since $Q_1^{x=0}$ contains some nonzero vector, we also see that $a'\le a-1$. So $a' = a-1$ and $Q_1^{x=0} = \{(0,1)\}$ (and also, evidently, $n=5$). To finish the argument, consider the partition $R=(R_i)_{i=1}^4$ obtained by taking $Q$ and moving $(0,1)$ from $Q_1$ to $Q_2$. The resulting lattice $\cL_R\subseteq \ZZ^4$ has Gale diagram $$G_R = \{(1,a-1),(-1,d),(-1,1-a),(1,-d)\},$$ which lies on two lines. Now, as in the first paragraph, $\deg I_R \ne \deg I_{\cL}$.
\end{proof}

\subsection{When equality holds in \eqref{ineqB}}
\label{sec:reg=degn=4}

In this subsection, we establish which lattices $\cL_0$ with $n=4$ satisfy $\reg I_{\cL_0}=\deg I_{\cL_0}$. One direction is provided by \cite{PS}*{Remark 7.9}, which states that if $\reg I_{\cL_0}=\deg I_{\cL_0}$ and $G_0$ is a Gale diagram for $\cL_0$ chosen so that the unit square is a syzygy quadrangle attaining the regularity, then up to dihedral symmetries, $G_0$ equals either
\begin{equation}
\label{reg=degGales}
    \{(1,1),(a,-b),(-1,-1),(-a,b)\}, \quad \text{or}\quad\{(1,a),(1,-b),(-1,-a),(-1,b)\},
\end{equation} for $a,b\ge 1$. In particular, if $\reg I_{\cL_0}=\deg I_{\cL_0}$, then there exists some Gale diagram of one of those two forms. We prove the converse to this statement (also implicit in \cite{PS}), as stated in \cref{lem:deg=a+b}.

\begin{lemma}\label{lem:deg=a+b}
Let $\cL_0\subseteq\ZZ^4$ be a lattice such that there exists a $G_0$ as in \eqref{reg=degGales}. Then $\deg I_{\cL_0} = a+b$, $\reg I_{\cL_0} = \deg I_{\cL_0}$, and the unit square is a syzygy quadrangle attaining the regularity.
\end{lemma}
\begin{proof}
Let $\cL_0$ be one of the listed lattices. The fact that $\deg I_{\cL_0} = a+b$ follows immediately from the formula in \cite{OPVV}*{Theorem 4.6}. By \cref{prop:PS-syz-quads}, the unit square is a syzygy quadrangle of $I_{\cL_0}$. Let $C$ be the multidegree for which the unit square is $P_C$. Then $\deg C = a+b+2$ in both cases of \eqref{reg=degGales}, so $\reg I_{\cL_0} \ge a+b$. On the other hand, $\deg I_{\cL_0} = a+b$, so $\reg I_{\cL_0} \ge \deg I_{\cL_0}$. By \cref{thm:reduction}, we are done.
\end{proof}

This result also has implications for lattice ideals of maximal regularity in arbitrary dimension. 

\begin{definition}
A reduction datum $(\cL,G,Q)$ is called \textit{perfectly balanced} if $$\sum_{\bb \in Q_1\cup Q_3} \bb= \sum_{\bb\in Q_2 \cup Q_4} \bb= 0.$$
\end{definition}

Then \cref{lem:deg=a+b} has the following corollary.

\begin{corollary}
Suppose $I_\cL$ is not Cohen--Macaulay and $\reg I_\cL = \deg I_\cL - 1$. Then there exists a perfectly balanced reduction datum $(\cL, G, Q)$ for which the unit square attains the regularity.
\end{corollary}

\subsection{Analyzing inequality \eqref{ineqC}}
\label{sec:deg=deg}

In this subsection we prove \cref{prop:deg-1-iff}, which characterizes when $\deg I_Q = \deg I_\cL - 1$ and will help prove \cref{thm:main}.

Fix a reduction datum $(\cL, G, Q)$. Throughout this section, $i$ and $j$ are indices such that $\{i,j\}$ is $\{1,3\}$ or $\{2,4\}$.

\begin{definition}
Define $A_i$ (resp. $B_i$) to be the set of $\bu \in \mathbb Z^2$ for which there exist $\bv_1,\bv_{-1}\in Q_i$ (resp. $\bv_1\in Q_i, \bv_{-1}\in Q_j$) such that $\bv_1\cdot\bu=-\bv_{-1}\cdot\bu=1$, and $\bb\cdot\bu=0$ for all $\bb\in (Q_i\cup Q_j)\setminus\{\bv_1,\bv_{-1}\}$. Also define $C_i \coloneqq A_i\cup B_i$. 
\end{definition}

\begin{definition}\label{def:simple}
Suppose that $\sum_{\bb\in Q_i\cup Q_j}\bb=0$. We say the reduction datum $(\cL, G, Q)$ is \textit{$i$-simple} if there are vectors $\bv,\bw\in \ZZ^2$ such that $\det(\bv,\bw) = 1$, the set of nonzero vectors in $Q_i$ is $\{\bv,\bw\}$, and the set of nonzero vectors in $Q_j$ is either $\{-\bv,-\bw\}$ or $\{-\bv-\bw\}$. We say $(\cL, G, Q)$ is \textit{$\{i,j\}$-simple} if it is $i$-simple or $j$-simple. If the reduction datum is understood, we also refer to $G$ as being \textit{$i$-simple} or \textit{$\{i,j\}$-simple}.
\end{definition}

\begin{proposition}
\label{bigdegiff}
Suppose $\sum_{\bb \in Q_i\cup Q_j}\bb = 0$ and $\frp = \langle y_i,y_j\rangle$ is an associated prime of $J_Q$ (see \cref{reductionprocess}). Then the following are equivalent:
\begin{enumerate}
    \item $\frp$ is the $\frp$-primary part of $J_Q$,
    \item there are distinct $f,g\in \Phi$ and distinct $f',g'\in \Phi$ for which $y_if-y_ig, y_jf'-y_jg'\in J_Q$,
    \item there are $\bu\in C_i$ and $\bu'\in C_j$ with $\bu\ne \pm \bu'$, 
    \item $B_i$ contains at least two distinct vectors,
    \item $(\cL, G, Q)$ is $\{i,j\}$-simple. 
\end{enumerate}
\end{proposition}
\begin{proof} Without loss of generality, suppose $i=2, j=4$. Let $\Phi = \{y_1^ay_3^b \mid a,b\ge 0\}$ be a multiplicatively closed subset of $\Bbbk[y_1,\dots,y_4]$, and let $\mathscr M$ denote the full set of monomials in $\Bbbk[y_1,\dots,y_4]$. \\

    \noindent (1)$\iff$(2): It is clear by \cite{DMM}*{Lemma 2.9, Theorem 2.14} that (2)$\implies$(1), since in this case, $y_2g/f\sim_{(J_Q)_\Phi} y_2$ and similarly for $y_4$, so $y_2$ and $y_4$ both appear in the $\frp$-primary part of $J_Q$. Also by \cite{DMM}*{Lemma 2.9, Theorem 2.14}, (1) implies that there is some $\bw = (w_1,w_2,w_3,w_4) \ne (0,1,0,0)$ such that $y_2 - \by^{\bw}\in (J_Q)_\Phi$, $w_2\ge 1$, $w_4\ge 0$. The condition that $\sum_{\bb \in Q_2\cup Q_4}\bb = 0$ implies that $(w_2-1)+w_4 = 0$, so $w_2=1$ and $w_4=0$; in particular $\by^\bw = y_2g/f$ for some $g,f\in \Phi$. For some $p\in \Phi$, then, we have $y_2pf - y_2pg \in J_Q$. The same argument applies to $y_4$. \\
    
    \noindent (2)$\implies$(3): Since $y_2f-y_2g \in J_Q$, it follows that $C_2$ is nonempty. Similarly $C_4$ is nonempty. There is some minimal $m > 0$, monomials $\beta_1,\cdots,\beta_m\in \mathscr M$, and vectors $\bu_1,\cdots,\bu_m\in \ZZ^2$ for which $$y_2f-y_2g = \sum_{\ell=1}^m \beta_\ell\cdot \phi_Q(\bx^{(B\bu_\ell)^+}-\bx^{(B\bu_\ell)^-}).$$ Suppose towards a contradiction that $|C_2\cup (-C_4)| = 1$, so that there is some vector $\bu \in \ZZ^2$ for which $C_2=\{\bu\}$, $C_4 = \{-\bu\}$. By the minimality of $m$, every $\bu_\ell$ must be in $C_2$ or $C_4$, so $\bu_\ell = \pm \bu$ for each $\ell$. Then $y_2f-y_2g = P\cdot \phi_Q(\bx^{(B\bu)^+}-\bx^{(B\bu)^-})$ for some polynomial $P$ (namely $\sum \beta_\ell$), which is absurd since $-\bu \in C_4$ but $y_2f-y_2g$ has no monomial terms divisible by $y_4$. 
    
    It follows that $|C_2\cup (-C_4)|\ge 2$, so there exist $\bu \in C_2$ and $\bu'\in C_4$ with $\bu \ne \pm \bu'$.\\
    
    \noindent (3)$\implies$(4): If $\bu\in B_2$ and $\bu'\in B_4$, we are done. Up to symmetries, there are two remaining cases:
\begin{itemize}
    \item $\bu \in A_2, \bu'\in A_4$: Let $\bv_1,\bv_{-1}\in Q_2$ with $\bv_1\cdot \bu = -\bv_{-1}\cdot \bu = 1$. The condition $\bu'\in A_4$ implies that $\bv_1,\bv_{-1}$ are both orthogonal to $\bu'$, and so $\bv_1,\bv_{-1}$ are collinear, but this is impossible since both are in $Q_2$ and lie on opposite sides of $\bu$. So this case does not occur. 
    \item $\bu \in A_2, \bu'\in B_4$: Then there are vectors $\bv_1,\bv_{-1}\in Q_2$ with $\bv_1\cdot \bu =-\bv_{-1}\cdot \bu = 1$ with $\bu$ orthogonal to all other vectors of $Q_2\cup Q_4$. Similarly there are vectors $\bw_{-1}\in Q_2$, $\bw_1\in Q_4$ for which $\bw_1\cdot \bu'=-\bw_{-1}\cdot \bu'=1$ and $\bu'$ is orthogonal to all other vectors of $Q_2\cup Q_4$. Thus $\{\bv_{\pm 1},\bw_{\pm 1}\}$ are the only nonzero vectors in $Q_2\cup Q_4$, since all other vectors are orthogonal to both $\bu$ and $\bu'$.
    
    If $\bw_{-1}$ is distinct from both $\bv_1$ and $\bv_{-1}$, then $\bv_1$ and $\bv_{-1}$ are both orthogonal to $\bu'$ and so are collinear, but this is impossible since $\bv_1$ and $\bv_{-1}$ lie on opposite sides of $\bu$. So $\bw_{-1} = \bv_1$ or $\bw_{-1}=\bv_1$. If $\bw_{-1} = \bv_1$, it is straightforward to verify that $-\bu-\bu', -\bu' \in B_2$. If $\bw_{-1}=\bv_{-1}$, one similarly verifies that $\bu-\bu', -\bu'\in B_2$.
    \end{itemize}
    So in all cases, there are two distinct vectors in $B_2$.\\
    
    \noindent (4)$\implies$(2): Since there are at least two distinct vectors in $B_2$, there are monomials $p,q,r,s\in \Phi$ for which $y_2p-y_4q, y_2r-y_4s \in J_Q$ and $p/r \ne q/s$. Writing 
    \begin{align*}
        y_2ps-y_2qr &= s(y_2p-y_4q)-q(y_2r-y_4s),\\
        y_4ps-y_4qr &= r(y_2p-y_4q)-p(y_2r-y_4s)
    \end{align*}gives the desired elements of $J_Q$.\\
    
    \noindent (4)$\implies$(5): Let $\bu,\bu'\in B_2$ be distinct. By definition of $B_2$, there exist $\bv_1,\bw_1\in Q_2$ and $\bv_{-1},\bw_{-1}\in Q_4$ such that $\bv_\eps \cdot \bu = \bw_\eps \cdot \bu'=\eps$ for $\eps = \pm 1$, and $\bb\cdot\bu=0$ for all $\bb\in (Q_2\cup Q_4)\setminus\{\bv_{\pm 1}\}$, and $\bb\cdot\bu'=0$ for all $\bb\in (Q_2\cup Q_4)\setminus\{\bw_{\pm 1}\}$. Note that $\bv_{\pm 1},\bw_{\pm 1}$ are the only nonzero vectors in $Q_2\cup Q_4$, since all other vectors equal $\llb 0,0\rrb=0$.
    
    Note that $\bu,\bu'$ are $\QQ$-linearly independent, so we can uniquely represent vectors $\bv\in\ZZ^2$ as ordered pairs $\llb\bv\cdot\bu,\bv\cdot\bu'\rrb.$ If $\bv_1\neq\bw_1$, then $\bv_1=\llb 1,0\rrb$ and $\bw_1=\llb 0,1\rrb$, and if $\bv_1=\bw_1$, then $\bv_1=\bw_1=\llb 1,1\rrb$. Similarly, if $\bv_{-1}\neq\bw_{-1}$, then $\bv_{-1}=\llb -1,0\rrb$ and $\bw_{-1}=\llb 0,-1\rrb$, and if $\bv_{-1}=\bw_{-1}$, then $\bv_{-1}=\bw_{-1}=\llb -1,-1\rrb$. It is straightforward to check that in three of these cases, the Gale diagram $G$ is $\{2,4\}$-simple. The remaining case, where $\bv_1=\bw_1$ and $\bv_{-1}=\bw_{-1}$, cannot occur if $\frp$ is an associated prime of $J_Q$, by \cref{lem:J-contains-line}.\\
    
    \noindent (5)$\implies$(4): Suppose without loss of generality that $Q_2$ contains exactly two nonzero vectors $\bv=(v_1,v_2)$ and $\bw=(w_1,w_2)$, and suppose without loss of generality that $\det(\bv,\bw)=1$, so $v_1w_2-v_2w_1=1$. Let $\bu_1=(w_2,-w_1)$ and $\bu_2=(-v_2,v_1)$ be nonzero vectors. Note that $\bv\cdot\bu_1=1$, $\bw\cdot\bu_1=0$, $\bv\cdot\bu_2=0$, and $\bw\cdot\bu_2=1$. Furthermore, either $-\bv-\bw$ is the only nonzero vector in $Q_4$, or $-\bv,-\bw$ are the only nonzero vectors in $Q_4$. In either case, we see that $\bu_1,\bu_2\in B_2$.
\end{proof}

We can now prove our main result comparing $\deg I_Q$ to $\deg I_\cL$. 
\begin{proposition}
\label{prop:deg-1-iff}
Suppose that $(\cL,G, Q)$ is perfectly balanced. Then $\deg I_Q=\deg I_\cL-1$ if and only if for either $\{i,j\}=\{1,3\}$ or $\{i,j\}=\{2,4\}$, the vectors in $G\setminus(Q_i\cup Q_j)$ all lie on a single line passing through the origin, and $(\cL, G,Q)$ is $\{i,j\}$-simple.
\end{proposition}
\begin{proof}
Recall that $\deg I_\cL=\deg J_Q$. We see that $\deg I_Q=\deg J_Q-1$ if and only if there is exactly one primary component $\frq$ of $J_Q$ for which $\codim\frq=2$ and $y_1y_2y_3y_4\in\frp \coloneqq \sqrt{\frq}$, and furthermore $\deg\frq=1$. By \cref{lem:J-contains-line}, this is equivalent to having $\frp = \langle y_i,y_j\rangle$ and $G\setminus(Q_i\cup Q_j)$ contained in a closed half-space passing through the origin for some $\{i,j\}$. Without loss of generality take $\{i,j\}=\{2,4\}$. The condition $\deg \frq = 1$ is equivalent to $\frq = \frp$, which by \cref{bigdegiff} is equivalent to $G$ being $\{2,4\}$-simple. The condition that $Q_{1}\cup Q_{3}$ is contained in a closed half-space is equivalent to $Q_{1}\cup Q_{3}$ being contained in a line, since $\sum_{\bb \in Q_{1}\cup Q_{3}}\bb = 0$. 
\end{proof}

\subsection{When equality holds in \eqref{ineqA}}
\label{sec:reg=reg}

In this subsection, we use \cref{cor:syzygy-degrees} to describe the syzygy quadrangles (and their corresponding total degrees) of $I_Q$ in terms of those of $I_\cL$, which allows us to determine when equality holds in \eqref{ineqA}.

Fix a reduction datum $(\cL, G, Q)$.
\begin{proposition}
\label{newsyzquadunitsquare}
Suppose that $(\cL, G, Q)$ is perfectly balanced, and that for either $\{i,j\}=\{1,3\}$ or $\{i,j\}=\{2,4\}$, the following two properties hold:
\begin{itemize}
    \item the vectors in $G\setminus (Q_i\cup Q_j)$ all lie on a single line passing through the origin, and
    \item $(\cL,G, Q)$ is $\{i,j\}$-simple in the sense of \cref{def:simple}.
\end{itemize}
Then the following are true:
\begin{enumerate}
    \item Each syzygy quadrangle of $I_\cL$ is a syzygy quadrangle of $I_Q$, with the same total degree.
    \item There exists exactly one syzygy quadrangle $P_G$ of $I_{Q}$ that is not a syzygy quadrangle of $I_\cL$.
    \item The total degree of $P_{G}$ is at least the total degree of the unit square. Equality holds if and only if up to dihedral symmetries, $G_Q$ is of one of the forms in \eqref{reg=degGales}.
    \item If up to dihedral symmetries, $G_Q$ is of one of the forms in \eqref{reg=degGales}, then the unit square is a syzygy quadrangle of $I_\cL$ that attains the regularity of $I_\cL$.
\end{enumerate}
\end{proposition}

\begin{proof}[Proof of \cref{newsyzquadunitsquare}(1)]
Let $[\bv,\bw]$ be an arbitrary syzygy quadrangle of $I_\cL$, where as in the proof of \cref{lem:eqdeg}, we may assume that $\bv,\bw$ either both lie in the first closed quadrant or both lie in the second closed quadrant, and that $\det(\bv,\bw)=1$. Without loss of generality, we assume that $\bv,\bw$ both lie in the first closed quadrant; the argument when they both lie in the second closed quadrant is similar. If $\{i,j\}=\{1,3\}$, then all the vectors in $Q_2\cup Q_4$ lie on a line passing through the origin, so it follows from \cref{lem:preserve-syz-quad-and-degree} that $[\bv,\bw]$ is also a syzygy quadrangle with respect to $I_{Q}$, with the same total degree.

Now, suppose that $\{i,j\}=\{2,4\}$, so $G$ is $\{2,4\}$-simple, and either $Q_2$ or $Q_4$ contains exactly two nonzero vectors. Without loss of generality, suppose that $Q_2$ contains exactly two nonzero vectors, which we denote $\bb$ and $\bc$, so that $\bb_2'=\bb+\bc$ and the set of nonzero vectors in $Q_4$ equals either $\{-\bb,-\bc\}$ or $\{-\bb-\bc\}$. Since $[\bv,\bw]$ is a syzygy quadrangle of $I_\cL$, by \cref{prop:PS-syz-quads}, either $-\bb\cdot\bv,\bb\cdot\bw>0$ or $-\bc\cdot\bv,\bc\cdot\bw>0$. Without loss of generality, suppose $-\bb\cdot\bv,\bb\cdot\bw>0$. By \cref{lem:preserve-syz-quad-and-degree}, it suffices to show that
\[Q_2\subseteq\{\bu\in\RR^2\mid\text{$\bu\cdot\bv\le 0$ and $\bu\cdot\bw\ge 0$}\},\qquad Q_4\subseteq\{\bu\in\RR^2\mid\text{$\bu\cdot\bv\ge 0$ and $\bu\cdot\bw\le 0$}\}.\]
Assume otherwise, so that either $\bc\cdot\bv>0$ or $\bc\cdot\bw<0$. If $\bc\cdot\bv>0$, then we must also have $\bc\cdot\bw>0$. However, this contradicts that $\det(\bv,\bw)=\abs{\det(\bb,\bc)}=1$. Similarly, if $\bc\cdot\bw<0$, then $\bc\cdot\bv<0$, so again we obtain a contradiction.
\end{proof}

\begin{proof}[Proof of \cref{newsyzquadunitsquare}(2)]
Let $\ell\in\{i,j\}$ be such that $Q_\ell$ contains exactly two nonzero vectors, which are denoted $\bv$ and $\bw$. Let $\bv',\bw'\in\ZZ^2$ be the images of $\bv,\bw$ when rotated $\pi/2$ radians counterclockwise about the origin, respectively. Let $P_{G}$ denote the primitive parallelogram $[\bv',\bw']$, considered up to lattice translations. Note that if $Q_i$ and $Q_j$ each contains exactly two nonzero vectors, the parallelogram $P_{G}$ does not depend on whether $\ell=i$ or $\ell=j$.

Without loss of generality, suppose $\{i,j\}=\{1,3\}$, and that $Q_1$ contains exactly two nonzero vectors $\bb,\bc$, where $\abs{\det(\bb,\bc)}=1$. Let $[\bv,\bw]$ be an arbitrary syzygy quadrangle of $I_Q$, where $\bv,\bw$ either both lie in the first closed quadrant or both lie in the second closed quadrant, and $\det(\bv,\bw)=1$. Suppose $[\bv,\bw]$ is not a syzygy quadrangle of $I_\cL$. If $\bv,\bw$ both lie in the first closed quadrant, then it follows from \cref{prop:PS-syz-quads} that $[\bv,\bw]$ is also a syzygy quadrangle of $I_\cL$, a contradiction. Thus, $\bv,\bw$ both lie in the second closed quadrant. Since $[\bv,\bw]$ is a syzygy quadrangle of $I_Q$, we must have $(\bb+\bc)\cdot\bv>0$ and $(\bb+\bc)\cdot\bw<0$, and since $[\bv,\bw]$ is not a syzygy quadrangle of $I_\cL$, we have that either $\bb\cdot\bv\le 0$ or $\bb\cdot\bw\ge 0$, and similarly for $\bc$. This implies that either $\bb\cdot\bv\le 0$ and $\bc\cdot\bw\ge 0$, or $\bb\cdot\bw\ge 0$ and $\bc\cdot\bv\le 0$. Without loss of generality, suppose $\bb\cdot\bv\le 0$ and $\bc\cdot\bw\ge 0$.

Writing $\bb,\bc,\bv,\bw$ as $2\times 1$ column vectors, let
\[\begin{bmatrix}\bb^T\\\bc^T\end{bmatrix}\begin{bmatrix}\bv&\bw\end{bmatrix}=A=\begin{bmatrix}a_{11}&a_{12}\\a_{21}&a_{22}\end{bmatrix}\in\ZZ^{2\times 2}.\]
We have $\abs{\det A}=1$, and $-a_{11},a_{22}\ge 0$ and $a_{11}+a_{21},-a_{12}-a_{22}\ge 1$. Thus,
\begin{align*}
\pm 1=\det A&=a_{11}a_{22}-a_{12}a_{21}\\
&=a_{11}a_{22}+((-a_{12}-a_{22})+a_{22})((a_{11}+a_{21})-a_{11})\\
&=(-a_{12}-a_{22})(a_{11}+a_{21})+(-a_{11})(-a_{12}-a_{22})+a_{22}(a_{11}+a_{21})\\
&\ge 1+0+0.
\end{align*}
Thus, we have $a_{11}+a_{21}=-a_{12}-a_{22}=1$ and $a_{11}=a_{22}=0$. Letting $\bb',\bc'$ denote the images of $\bb,\bc$ after a rotation of $\pi/2$ about the origin, it follows that $\bb'=\bv$ and $\bc'=\bw$. Thus, $[\bv,\bw]=[\bb',\bc']=P_{G}$. Conversely, it is easy to see that $P_{G}$ is a syzygy quadrangle of $I_Q$ but not of $I_\cL$.
\end{proof}

\begin{proof}[Proof of \cref{newsyzquadunitsquare}(3)]
Without loss of generality, suppose $\{i,j\}=\{1,3\}$, and that $Q_1$ contains exactly two nonzero vectors $\bb,\bc$, where $\det(\bb,\bc)=1$. Let $\bb=(b_1,b_2)$ and $\bc=(c_1,c_2)$. Furthermore, let the sum of the vectors in $Q_2$ equal $(-a_1,a_2)$, where $a_1,a_2\ge 1$. Using \cref{cor:syzygy-degrees}, it is easy to see that for $I_Q$, the total degree corresponding to the unit square equals $a_1+a_2+b_1+b_2+c_1+c_2$. Furthermore, again by \cref{cor:syzygy-degrees}, we see that the total degree corresponding to the syzygy quadrangle $P_{G}=[(-b_2,b_1),(-c_2,c_1)]$ of $I_Q$ equals
\[(\bb+\bc)\cdot(-b_2,b_1)+(-a_1,a_2)\cdot(-b_2-c_2,b_1+c_1)+(-\bb-\bc)\cdot(-c_2,c_1)+(a_1,-a_2)\cdot 0,\]
which equals $-b_2c_1+b_1c_2+a_1(b_2+c_2)+a_2(b_1+c_1)+b_1c_2-b_2c_1=a_1(b_2+c_2)+a_2(b_1+c_1)+2$. The desired inequality is
\[a_1+a_2+(b_1+c_1)+(b_2+c_2)\le a_1(b_2+c_2)+a_2(b_1+c_1)+2,\]
which can be rearranged as
\[0\le (a_1-1)(b_2+c_2-1)+(a_2-1)(b_1+c_1-1).\]
This inequality holds since either $\bb$ or $\bc$ lies in the first open quadrant. Equality holds if and only if $a_1=1$ or $b_2+c_2=1$, and $a_2=1$ or $b_1+c_1=1$, as desired.
\end{proof}

\begin{proof}[Proof of \cref{newsyzquadunitsquare}(4)]
By part (3), the total degree corresponding to the syzygy quadrangle $P_{G}$ of $I_Q$ equals the total degree corresponding to the unit square. Combining this fact with parts (1) and (2) of \cref{newsyzquadunitsquare}, we see that $\reg I_\cL=\reg I_Q$, since from \cite{PS}, we know that the regularity of $I_\cL$ equals $-2$ plus the maximum total degree corresponding to a syzygy quadrangle of $I_\cL$, and similarly for $I_Q$. By \cref{lem:deg=a+b}, the unit square is a syzygy quadrangle of $I_Q$ that attains the regularity of $I_Q$. Since the unit square is a syzygy quadrangle of both $I_\cL$ and $I_Q$, with the same total degree, the unit square also attains the regularity of $I_\cL$.
\end{proof}

For our last result, we assume that the unit square is a syzygy quadrangle of $I_\cL$ that attains the regularity of $I_\cL$. Strictly speaking, it is possible to circumvent \cref{prop:same-reg-iff} in the proof of \cref{prop:non-cm-main-result}, but we include \cref{prop:same-reg-iff} anyway since it fits in naturally with the flow of our results.

\begin{proposition}\label{prop:same-reg-iff}
Suppose $(\cL, G, Q)$ satisfies the hypotheses of \cref{newsyzquadunitsquare}, and furthermore that
\begin{itemize}
    \item the unit square is a syzygy quadrangle attaining the regularity of $I_\cL$, and
    \item $\deg I_Q = \deg I_\cL - 1$.
\end{itemize}
Then, $\reg I_{Q}=\reg I_\cL$ if and only if up to dihedral symmetries, $G_Q$ is of one of the forms in \eqref{reg=degGales}.
\end{proposition}

\begin{proof}
Since $G$ is chosen so that the unit square is a syzygy quadrangle attaining the regularity of $I_\cL$, it follows from \cref{newsyzquadunitsquare} that $P_{G}$ is a syzygy quadrangle of $I_Q$ attaining the regularity (of $I_Q$), and that $\reg I_Q=\reg I_\cL$ if and only if up to dihedral symmetries, $G_Q$ is of one of the desired forms.
\end{proof}

\section{Non-Cohen--Macaulay lattice ideals of codimension 2: Proof of \cref{thm:main}}
\label{sec:finalproof}

In this section we specialize the results of \cref{sec:analyzingreduction} to the case of toric ideals with maximal regularity to prove \cref{thm:main}.

\begin{proposition}\label{prop:preliminary-non-cm-main-result}
Suppose that $\cL$ is saturated with Gale diagram containing at least $5$ nonzero vectors. The toric ideal $I_\cL$ has maximal regularity if and only if there exists a reduction datum $(\cL, G, Q)$, satisfying the following properties for either $\{i,j\}=\{1,3\}$ or $\{i,j\}=\{2,4\}$:
\begin{itemize}
    \item the vectors in $G\setminus(Q_i\cup Q_j)$ all lie on a single line passing through the origin,
    \item $(\cL, G, Q)$ is $\{i,j\}$-simple, and
    \item up to dihedral symmetries, the set $G_Q$ is of one of the forms in \eqref{reg=degGales}.
\end{itemize}
\end{proposition}

\begin{proof}
As mentioned in \cref{sec:preliminaries}, it suffices to prove the statement under the assumption that all the Gale vectors are nonzero, so we assume that $G$ consists only of nonzero vectors. First, suppose that $I_\cL$ has maximal regularity. Since $I_\cL$ is not Cohen--Macaulay, and is therefore not a complete intersection, and $I_\cL$ is toric, \cite{PS}*{Proposition 7.10} implies that $G$ does not lie on two lines. Applying \cref{cor:eqreg,prop:regdeg}, we conclude that we can choose a reduction datum $(\cL, G, Q)$ in such a way that $\reg I_\cL=\reg I_Q=\deg I_Q=\deg I_\cL-1$. Then the unit square is a syzygy quadrangle of $I_Q$ that attains the regularity of $I_Q$, so \cite{PS}*{Remark 7.9} implies that $G_Q$ is of one of the forms in \eqref{reg=degGales}. In particular, this means that $(\cL,G,Q)$ is perfectly balanced. Then, the remaining desired properties of $G$ follow from \cref{prop:deg-1-iff}.

Now, suppose that there exists a reduction datum that satisfies the given properties. Per \eqref{reductionchain}, we have $\reg I_\cL\le\reg I_Q\le\deg I_Q\le\deg I_\cL$. By \cref{lem:deg=a+b}, we have $\reg I_Q=\deg I_Q$, and the unit square is a syzygy quadrangle of $I_Q$ that attains the regularity of $I_Q$. The third condition implies that $(\cL,G,Q)$ is perfectly balanced, so by \cref{prop:deg-1-iff} and the first two conditions, $\deg I_Q=\deg I_\cL-1$. Finally, $\reg I_\cL=\reg I_Q$ by \cref{prop:same-reg-iff}. Another way to see that $\reg I_\cL=\reg I_Q$ is to use the fact the unit square is a syzygy quadrangle of $I_\cL$ that attains the regularity of $I_\cL$, and a syzygy quadrangle of $I_Q$ that attains the regularity of $I_Q$ but has the same corresponding total degree with respect to $I_\cL$ and $I_Q$.
\end{proof}

For the next result, we use \cref{newsyzquadunitsquare} to reformulate the previous proposition in a way that does not require the underlying Gale diagram to be chosen so that the unit square attains the regularity. In particular, it also tells us that if $I_\cL$ is not Cohen--Macaulay and has maximal regularity, then $\cL$ has at most $6$ nonzero Gale vectors.

\begin{proposition}\label{prop:non-cm-main-result}
Suppose that $\cL$ is saturated, has rank $2$, and the Gale diagram of $\cL$ contains at least $5$ nonzero vectors. Then, $I_\cL$ is a non-Cohen--Macaulay ideal of maximal regularity if and only if there exists a $\{2,4\}$-simple reduction datum $(\cL, G, Q)$ such that $Q_1=\{(1,1)\}$ and $Q_3=\{(-1,-1)\}$.
\end{proposition}

\begin{proof}
Suppose first that such a datum exists. It then follows from \cref{lem:ci-imbalanced,prop:non-cm-open-quadrants} that $I_\cL$ is not Cohen--Macaulay. Furthermore, by \cref{newsyzquadunitsquare}, the unit square is a syzygy quadrangle of $I_\cL$ that attains the regularity of $I_\cL$. So the desired result follows from \cref{prop:preliminary-non-cm-main-result}.

For the other direction, suppose $I_\cL$ is non-Cohen--Macaulay with maximal regularity. Then there is a reduction datum $(\cL,G,Q)$ satisfying the conditions in \cref{prop:preliminary-non-cm-main-result} and that $I_\cL$ has maximal regularity. Suppose first that $G_Q$ equals $\{(1,1),(a,-b),(-1,-1),(-a,b)\}$ up to dihedral symmetries for some $a,b>0$; after applying a suitable dihedral symmetry (which preserves the fact that the unit square is a syzygy quadrangle attaining the regularity, as well as the conditions in \cref{prop:preliminary-non-cm-main-result}), we may suppose that $G_Q=\{(1,1),(a,-b),(-1,-1),(-a,b)\}$. Since $G$ intersects each open quadrant, $Q_1$ must consist of $(1,1)$ and some number of zero vectors, and $Q_3$ must consist of $(-1,-1)$ and some number of zero vectors. Thus, $G$ is not $\{1,3\}$-simple, so $\{i,j\}=\{2,4\}$ and $G$ is $\{2,4\}$-simple. Moving all zero vectors from $Q_1$ and $Q_3$ to either $Q_2$ or $Q_4$, we obtain a partition $Q'$ for which $(\cL,G,Q')$ has the desired properties.

Now, suppose that $G_Q$ equals $\{(1,a),(1,-b),(-1,-a),(-1,b)\}$ up to dihedral symmetries for some $a,b>0$; again, we may assume that $G_Q=\{(1,a),(1,-b),(-1,-a),(-1,b)\}$. After reflecting if necessary, we may also suppose $\{i,j\}=\{2,4\}$. This forces $Q_1$ to consist of $(1,a)$ and some number of zero vectors, and $Q_3$ to consist of $(-1,-a)$ and some number of zero vectors. We can move all the zero vectors to $Q_2$ or $Q_4$. We now multiply every vector in $G$ by the matrix
\[\begin{bmatrix}1&0\\-a+1&1\end{bmatrix}\in\GL_2(\ZZ)\]
to obtain a Gale diagram $G'$ of $\cL$. For all $1\le m\le 4$, let $Q_m'$ denote the set of images of the Gale vectors in $Q_m\subseteq G$, so that $G=Q_1'\cup Q_2'\cup Q_3'\cup Q_4'$, where $Q_1'=\{(1,1)\}$, and $Q_3'=\{(-1,-1)\}$. Note that each $Q_m'$ is contained in the $m$th closed quadrant and contains a vector in the $m$th open quadrant. Using the fact that $(\cL,G,Q)$ is $\{2,4\}$-simple, it is straightforward to check that $(\cL,G',Q')$ is also $\{2,4\}$-simple. The result follows.
\end{proof}

We are finally able to give a proof of \cref{thm:main}.

\begin{proof}[Proof of \cref{thm:main}]
The case when $I_\cL$ is a complete intersection is covered by the saturated lattices given in \cref{ci-lattice-ideals} (up to permutations), as discussed in \cref{cor:finitely-many-ci-lattice-ideals}. So, we suppose from now on that $I_\cL$ is not a complete intersection.

If $n'=3$, then $I_\cL$ is Cohen--Macaulay by \cref{prop:non-cm-open-quadrants}, but by \cref{prop:codim-2-CMnonCI-list}, there are no such saturated $\cL$ that give $\reg I_\cL=\deg I_\cL-1$.

If $n'=4$, we are done by \cref{cor:toric-n=4-iff}.

Suppose $n'=5$. By combining \cref{prop:codim-2-CMnonCI-list} and \cref{prop:non-cm-main-result}, we see that $I_\cL$ has maximal regularity if and only if there exists a Gale diagram $G$ of $\cL$ of the form
\[\{(-1,1),(1,-1),\bv,\bw,-\bv-\bw\},\]
where $\bv,\bw$ are vectors in the closed first quadrant such that $\abs{\det(\bv,\bw)}=1$. Moreover, for all such Gale diagrams, the corresponding lattice ideal is toric and not a complete intersection. If $\bu$ is an arbitrary visible lattice point (i.e., a lattice point with relatively prime coordinates) in the open second quadrant, there exists some $M\in\GL_2(\ZZ)$ with nonnegative entries such that $M\bu=\left[\begin{smallmatrix}-1\\1\end{smallmatrix}\right]$. It follows that $I_\cL$ has maximal regularity if and only if there exists a Gale diagram $G$ of the form
\[\{\bu,-\bu,(1,0),(0,1),(-1,-1)\},\]
where $\bu$ is a visible lattice point in the open second quadrant. After applying additional $\GL_2(\ZZ)$ transforms (which may permute $(1,0),(0,1),(-1,-1)$), we may only require $\bu$ to be a visible lattice point not lying on the coordinate axes or on the line $y=x$. The result follows.

The case where $n'=6$ follows by an argument similar to the one for $n'=5$. Finally, by \cref{prop:codim-2-CMnonCI-list} and \cref{prop:non-cm-main-result}, the assumption that $I_\cL$ is not a complete intersection means that equality is never achieved if $n'>6$.
\end{proof}

\section{Future work}\label{sec:future-work}
The methods of \cref{subsec:codim-2-CMnonCI} gave a complete classification of the lattices $\cL$ of rank $2$ for which $I_\cL$ is a Cohen--Macaulay lattice ideal such that $\reg I_\cL=\deg I_\cL-\codim I_\cL+1$. We believe with some more work, a similar classification could be obtained for lattices of higher rank. Empirical evidence suggests the following.
\begin{conjecture}
Suppose $\operatorname{char}\Bbbk=0$, and let $I_\cL$ be a nondegenerate Cohen--Macaulay lattice ideal such that $\reg I_\cL=\deg I_\cL-\codim I_\cL+1$. If $\codim I_\cL\ge 2$, then all minimal generators of $I_\cL$ have degree $2$.
\end{conjecture}

This says that the case of \cref{prop:general-CM-iff} with a generator of degree $\reg I \ge 3$ does not occur in the case of lattice ideals. As in \cref{cor:finitely-many-ci-lattice-ideals} and \cref{prop:codim-2-CMnonCI-list}, this would imply that up to the inclusion of nonzero Gale vectors and permuting coordinates, there are only finitely many possibilities for $\cL$. In fact, we have the following stronger conjecture.

\begin{conjecture}
Suppose $\operatorname{char}\Bbbk=0$, and let $I_\cL$ be a nondegenerate Cohen--Macaulay lattice ideal such that $\reg I_\cL=\deg I_\cL-\codim I_\cL+1$. If $\reg I_\cL\ge 3$ and $r=\codim I_\cL\ge 2$, then the minimal free resolution of $I_\cL$ is pure, with the form
  \[
  0\to S(-r-2)^1\to S(-r)^{\frac{(r-1)1}{r+1}\binom{r+2}{r}}\to\cdots
   \to S(-3)^{\frac{2(r-2)}{r+1}\binom{r+2}{3}}\to S(-2)^{\frac{1(r-1)}{r+1}\binom{r+2}{2}}\to I\to 0.
   \]
\end{conjecture}

Note that these conjectures are confirmed for complete intersections by \cref{prop:general-ci-ideals}, for ideals of codimension $2$ that are not complete intersections by \cref{prop:codim-2-CMnonCI-iff}, and for Gorenstein ideals of codimension $3$ by \cref{rmk:codim-3-gorenstein}.



\end{document}